\documentclass[10pt]{amsart}
\usepackage{amssymb}
\usepackage{bbm}
\usepackage[arrow,matrix]{xy}

\numberwithin{equation}{section}
\newtheorem{theorem}{Theorem}[section]
\newtheorem{corollary}[theorem]{Corollary}
\newtheorem{lemma}[theorem]{Lemma}
\newtheorem{proposition}[theorem]{Proposition}
\newtheorem{definition}[theorem]{Definition}
\newtheorem{remark}[theorem]{Remark}
\newtheorem{example}[theorem]{Example}

\makeatletter

\def\Ext{\operatorname {Ext}}

\def\deg{\operatorname {deg}}

\def\GrAut{\operatorname {GrAut}}
\def\H{\operatorname {H}}
\def\hdet{\operatorname {hdet}}
\def\det{\operatorname {det}}

\def\deg{\operatorname {deg}}

\def\inj{\operatorname {inj}}
\def\gldim{\operatorname {gldim}}
\def\id{\operatorname {id}}

\newcommand{\rmnum}[1]{\romannumeral #1}

\title[Nakayama automorphisms of double Ore extensions of Koszul regular algebras]
{\bf Nakayama automorphisms of double Ore extensions of Koszul regular algebras}

\author{Can  Zhu}
\address{C. Zhu\newline College of Science, University of Shanghai for Science and Technology, Shanghai
200093, China}
\email{czhu@usst.edu.cn}

\author{Fred Van Oystaeyen}
\address{F. Van Oystaeyen\newline Department of Mathematics and Computer Science, University of Antwerp, Middelheimlaan 1,
B-2020 Antwerp, Belgium} \email{fred.vanoystaeyen@ua.ac.be}

\author{Yinhuo Zhang}
\address{Y. Zhang\newline Department of mathematics and Statistics, University of Hasselt, 3590 Diepeenbeek, Belgium} \email{yinhuo.zhang@uhasselt.be}
\date{}

\begin{document}
\begin{abstract} Let $A$ be a Koszul Artin-Schelter regular algebra and $\sigma$  an algebra homomorphism from $A$ to $M_{2\times 2}(A)$. We compute  the Nakayama automorphisms of a trimmed double Ore extension $A_P[y_1, y_2; \sigma]$ (introduced in \cite{ZZ08}). Using a similar method, we also obtain the Nakayama automorphism of a skew polynomial extension $A[t; \theta]$,  where $\theta$ is  a graded algebra automorphism of $A$. These lead to a characterization of the Calabi-Yau property of  $A_P[y_1, y_2; \sigma]$, the skew Laurent extension $A[t^{\pm 1}; \theta]$ and $A[y_1^{\pm 1}, y_2^{\pm 1}; \sigma]$ with $\sigma$ a diagonal type.
\end{abstract}


%16S80 Deformations of rings
%16E40  (Co)homology of rings and algebras (e.g. Hochschild, cyclic, dihedral,
%etc.)
%16S37 Quadratic and Koszul algebras


%%% ----------------------------------------------------------------------
\maketitle
%%% ----------------------------------------------------------------------

\section*{Introduction}
%Since Ginzburg's original definition \cite{G06} and Brown-Zhang's paper %\cite{BZ08}, Calabi-Yau and twisted Calabi-Yau algebras attract more and more %attention due to their background and application in noncommutative %projective
%algebraic geometry, quantum algebras, representation theory and etc, see e.g. %\cite{B08, D09, IR08, K12, RRZ13, S08, YZ11}.

Nakayama automorphisms play an important role in  noncommutative algebraic geometry especially in noncommutative invariant theory \cite{CWZ14, LMZ14, RRZ14}.
Let $A$ be a Koszul Artin-Schelter regular algebra with  Nakayama automorphism $\nu$ in the sense of \cite{BZ08}. The Nakayama automorphism and Calabi-Yau property of Ore extensions and of skew polynomial extensions were studied in \cite{LWW12, BOZZ13, HVZ13, GK14, GYZ14, RRZ14}.  In this paper, we
compute  the Nakayama automorphisms of certain double Ore extension $A_P[y_1, y_2; \sigma]$ of $A$; the general notion of a double Ore extension was introduced by Zhang and Zhang in \cite{ZZ08}. Then we study the Calabi-Yau property of $A_P[y_1, y_2; \sigma]$,  a skew Laurent extension $A[t^{\pm 1}; \theta]$, where $\theta\in \GrAut(A)$, and $A[y_1^{\pm 1}, y_2^{\pm 1}; \sigma]$ with $\sigma$ a diagonal type.

It is well-known that a graded Ore extension of a Koszul algebra is also Koszul (see \cite[Corollary 1.3]{Ph12} for example). For a Koszul Artin-Schelter regular algebra, Van den Bergh proposed an effective method to compute the Nakayama automorphism through the Yoneda Ext algebra (see Proposition \ref{criteria} or \cite[Theorem 9.2]{VdB97}).  Inspired by these two facts, we first show the following:

\noindent {\bf Theorem 1.} (Theorem \ref{koszulity}) \emph{
Let $A$ be a Koszul algebra and  $B=A_P[y_1, y_2; \sigma]$ be a trimmed double Ore extension of $A$. Then, $B$ is a Koszul algebra.}

By describing the Yoneda Ext algebra, we are able to compute the Nakayama automorphism of a trimmed double Ore extension of a Koszul Artin-Schelter regular algebra.

\noindent {\bf Theorem 2.} (Proposition \ref{nakaab} and Theorem \ref{nakab}) \emph{
Let $A$ be a Koszul Artin-Schelter regular algebra with Nakayama automorphism $\nu$, and  $B=A_P[y_1, y_2; \sigma]$ a trimmed double Ore extension of $A$. Then,
\begin{enumerate}
\item The restriction of the Nakayama automorphism $\nu_{B}$ of
$B$ to $A$ equals $(\det_r\sigma)^{-1}\nu$, and
$$
\nu_{B}\begin{pmatrix}y_1 \\  y_2
\end{pmatrix}=(\hdet\sigma)\mathbb{P}^{-1}\begin{pmatrix}y_1 \\  y_2
\end{pmatrix},
$$
where $\det_r\sigma$ is an algebra automorphism induced by $\sigma$,  $\hdet\sigma\in M_2(\mathbbm{k})$ is determined by $\sigma$, and $\mathbb{P}\in M_2(\mathbbm{k})$ is determined by the data $P$  (see Equation \ref{matrixp}, Equation \ref{detdef} and Definition  \ref{hdett} for their definitions);\\
\item $B$ is Calabi-Yau if and only if $\det_r\sigma=\nu$ and $\hdet\sigma=\mathbb{P}$.
\end{enumerate}}

In a similar way, one can obtain the analogous results on the Nakayama automorphism and the Calabi-Yau property of the skew polynomial extension of a Koszul Artin-Schelter regular algebra (see  Proposition \ref{nakare} and Theorem \ref{orecy}).

In \cite[Theorem 6]{F05}, Farinati showed that the Calabi-Yau property is preserved by noncommutative localizations. Here, we characterize the Calabi-Yau property of the localization of both the skew polynomial extension with respect to the Ore set $\{t^i, i\in \mathbbm{N}\}$ (called the skew Laurent extension) and the iterated skew polynomial extension. The third main result reads as follows:

 \noindent {\bf Theorem 3.} (Theorem \ref{skewl} and Theorem \ref{skew2})
\emph{ Let $A$ be a Koszul Artin-Schelter regular algebra with  Nakayama automorphism $\nu$. \\
{\rm(1)}. The skew Laurent extension $A[t^{\pm 1}; \theta]$ of $A$
is Calabi-Yau if and only if there exists an integer $n$ such that
$\theta^n=\nu$ and the homological determinant $\hdet(\theta)$ of $\theta$ equals $1$.\\
{\rm(2)}. Given two automorphisms $\tau$ and $\xi$ of $A$, let, $Q=A[y_1^{\pm 1}, y_2^{\pm 1}; \sigma]$, where $\sigma=\text{diag}(\tau, \xi)$ is a map from $A$ to $M_{2\times 2}(A)$. Then, $Q$ is Calabi-Yau if and only if there exists two integers $m, n$ such that $\tau^m\xi^n=\nu$ and $\hdet(\tau)=\hdet(\xi)=1$.}

In fact, Part (2) of Theorem 3 is a special case of what is proved in Theorem \ref{skew2}. The aforementioned results and their proofs indicate that there exists a strong relation between the Nakayama automorphisms of those extensions and the homological determinants of the automorphisms which determine those extensions (see Theorem 2 and Theorem 3). The Nakayama automorphisms of the right coideal subalgebras of the quantized enveloping algebras were explicitly computed \cite{LW14}. In fact, those coideal subalgebras are special iterated Ore extensions. The general case of iterated Ore extensions and their relation with
double Ore extensions  were discussed in \cite{CLM11}.  So it would be interesting to study the Nakayama automorphism and the Calabi-Yau property of double Ore extensions and those of the localizations of iterated skew polynomial extensions in general.

The paper is organised as follows. In Section 1, we recall the definitions and the properties, including the relation between the Nakayama automorphism of a Koszul Artin-Schelter regular algebra and its Yoneda Ext algebra.  Section 2  prepares necessary means for computing the Nakayama automorphisms of trimmed double Ore extensions of Koszul algebras.

In Section 3, we mainly compute the Nakayama automorphism and study the Calabi-Yau property of trimmed double Ore extensions of Koszul Artin-Schelter regular algebras.
In Section 4, apart from what we mentioned in Theorem 3,  the Calabi-Yau property of the skew Laurent extensions and the Calabi-Yau property of a localization of iterated Ore extensions are studied. Necessary and sufficient conditions for those algebras to be Calabi-Yau are determined, see Theorem \ref{cyite}.

Throughout, $\mathbbm{k}$ is a field  and all algebras are $\mathbbm{k}$-algebras;
unadorned $\otimes$ means $\otimes_\mathbbm{k}$ and $^*$ always denotes the dual over $\mathbbm{k}$.

\section{Preliminaries}

An $\mathbb{N}$-graded algebra $A=\bigoplus_{i\geqslant0}A_i$ is called
 connected if $A_0=\mathbbm{k}$. By a graded algebra we mean a
locally finite graded algebra generated in degree $1$.  Let $A^e=A\otimes A^{op}$ denote the enveloping algebra of $A$. A module means a left (graded) module. The shifting of a graded module is denoted $(\,)$.
For a module $M$ over $A$, $^\varphi M$ stands for a twisted module by an algebra automorphism $\varphi$, where the
action is defined by $a\cdot m:=\varphi(a)m$. Similarly, $M^\varphi$ and $^1M^\varphi$ denote the twisted right module and the twisted bimodule respectively. %The
%category $\GrMod(A)$ consists of all graded left $A$-modules and
%graded homomorphisms of degree zero.  The group of all graded
%$k$-algebra automorphisms of $A$ is denoted by $\GrAut(A)$. %For $M,
%N\in\GrMod A$, let $M(n)$ be the shifted module of $M$ with
%$M(n)_i=M_{n+i}$, and
%$$\underline{\Hom}_A(M,
%N)=\bigoplus\limits_{i\in\mathbb{Z}}\Hom_{\GrMod(A)}(M, N(i)).$$
%%There is a forgetful functor $\GrMod A\rightarrow\Mod A.$
%Observe that in general $\underline{\Hom}_A(M, N)\subseteq\Hom_A(M,
%N)$, and the equality holds when $M$ is finitely generated. For any
%graded right $A$-module $M$ and left $A$-module $N$, the tensor
%product $M\otimes_AN$ is a graded space by defining $\deg(m\otimes
%n)=\deg(m)+\deg(n)$ for any $m \in M, n \in N$. $V^*$ stands for the
%vector space dual of $V$, and the Matlis dual
%$L^*=\underline{\Hom}(L,k)$ is the graded vector space dual for any
%graded vector space $L$.
%
%
%Now, we recall some basic concepts related to quadratic algebras and
%Koszul algebras.

Let $V$ be a finite-dimensional vector space, and
$T_\mathbbm{k}(V)$ be the tensor algebra with the usual grading. A connected graded algebra
$A=T_\mathbbm{k}(V)/\langle \, R \,\rangle$ is called a quadratic algebra if
$R$ is a subspace of $V^{\otimes 2}$. The homogeneous dual of $A$ is
then defined as $A^!=T_\mathbbm{k}(V^*)/\langle \,R^\perp \,\rangle$,
where
$$R^\perp=\{\lambda\in V^*\otimes V^*\mid \lambda(r)=0, \quad \text{for all}\quad r\in R\}.$$
Here, we identify $(V\otimes V)^*$ with $V^*\otimes V^*$ by
\begin{equation} \label{rperp} (\alpha\otimes \beta)(x\otimes y)=\alpha(x)\beta(y)
\end{equation}
for $\alpha, \beta\in V^*$ and $x, y\in V$. For more detail, see \cite{Sm96}.

\begin{definition} A quadratic algebra $A$ is called  Koszul if the
trivial $A$-module $_A\mathbbm{k}$ admits a projective resolution
\begin{equation*} \label{koszul} \cdots{\longrightarrow} P_n{\longrightarrow} P_{n-1} {\longrightarrow}
\cdots \longrightarrow P_1{\longrightarrow}
P_0{\longrightarrow}_A\mathbbm{k}{\longrightarrow}0
\end{equation*} such that $P_n$ is
generated in degree $n$ for all $n \geq 0.$
\end{definition}

For more detail about Koszul algebras and the Koszul duality, we refer
the reader to \cite[Ch.2]{PP05}. Now, we recall the definitions of an Artin-Schelter regular algebra, a Nakayama automorphism and a Calabi-Yau algebra.

\begin{definition}
A connected graded algebra $A$ is called  Artin-Schelter (AS, for short) Gorenstein of dimension $d$ with parameter $l$ for some integers $d$ and $l$, if
\begin{enumerate}
\item [(i)] $\inj.\dim(_AA)=\inj.\dim(A_A)=d$; and
\item [(ii)] $\Ext^i_A(\mathbbm{k}, A)\cong\Ext^i_{A^{op}}(\mathbbm{k}, A)\cong\begin{cases}0, & i \neq
d,  \\
 \mathbbm{k}(l), & i=d.
\end{cases}$
\end{enumerate}
If, in addition, $A$ has a finite global dimension, then $A$ is called
 AS-regular.
\end{definition}

\begin{definition} \cite{G06, BZ08} A graded algebra $A$ is called twisted Calabi-Yau of dimension $d$ if
\begin{enumerate}
\item  [(i)] $A$ is homologically smooth, i.e., $A$, as an $A^e$-module,
has a finitely generated projective resolution of finite length.
%$A$ has a bounded $A^e$-projective resolution consisting of finitely
%generated $A^e$- modules.
\item [(ii)]
$\Ext^i_{A^e}(A, A^e)\cong\begin{cases}0, & i \neq d
  \\
  A^{\nu}(l), & i=d
\end{cases}$ as  $A^e$-modules for some automorphism $\nu$ of $A$ and some integers $d, l$.
\end{enumerate}
The automorphism $\nu$ is called {\it the Nakayama automorphism} of $A$.
If, in addition, $A^{\nu}$ is isomorphic to $A$ as $A^e$-modules, or equivalently, $\nu$ is inner, then $A$ is called {\it Calabi-Yau of dimension $d$}. Ungraded Calabi-Yau algebras are defined similarly but without degree shift.
\end{definition}

Let $E$ be a Frobenius algebra. By definition, there is an isomorphism $\varphi: E\longrightarrow E^*$ of right $E$-modules. This is equivalent to the existence of a nondegenerate bilinear form, often called Frobenius pair, $\langle-, -\rangle: E\times E\rightarrow \mathbbm{k}$
such that $\langle ab, c\rangle=\langle a, bc\rangle$ for all $a, b ,c\in E$ (where the bilinear form is defined by
$\langle a, b\rangle:=\varphi(b)(a)$). By the nondegeneracy of the bilinear form, there exists an automorphism
$\mu$, unique up to an inner automorphism, such that
\begin{equation} \label{nakdef}
\langle a, b\rangle=\langle\mu(b), a\rangle
\end{equation}  for all $a, b\in E$.
Thus, $\varphi$ becomes an isomorphism of $E$-bimodules $^\mu E\cong E^*$.
The automorphism $\mu$ is usually called the Nakayama automorphism
of $E$.  For more detail, see \cite{Sm96}.

{\it Now, there are two notions of Nakayama automorphisms: one  for twisted Calabi-Yau algebras and one for
Frobenius algebras. We use $\nu$ for the former and $\mu$ for the latter if there is no confusion}.  In fact, the notion of a Nakayama automorphism in \cite{BZ08} can be defined for algebras
with finite injective dimension, and it coincides with the classical Nakayama automorphism of a Frobenius
algebra. But in this paper, we focus ourselves on twisted Calabi-Yau algebras (or equivalently, AS-regular algebras in
the connected graded case \cite[Lemma 1.2]{RRZ14}). It is well known that a connected graded algebra $A$ is AS-regular if and only if its Yoneda Ext algebra
is Frobenius \cite[Corollary D]{LPWZ08}. In this case, the two notions of Nakayama automorphisms will coincide in the sense of the  Koszul duality, see
Proposition \ref{criteria}. To get there, we need the following preparation.

Let $A=T_\mathbbm{k}(V)/\langle \,R \,\rangle$ be a
Koszul algebra. Then its Yoneda Ext algebra $E(A):=\bigoplus_{i\in \mathbb{N}}\Ext_A^i(\mathbbm{k}, \mathbbm{k})$ is isomorphic to $A^!=T_\mathbbm{k}(V^*)/\langle \, R^\perp \,\rangle$, see \cite{Sm96}.
For a  graded automorphism $\theta$ of $A$, we define a map
$\theta^*: V^*\rightarrow V^*$ by $\theta^*(f)(x)=f(\theta(x))$ for each $f\in V^*$ and $x\in V$. It is easy to see that $\theta^*$ induces a graded automorphism of $A^!$ because $\theta$ is assumed to preserve the relation space $R$.
We still use the notation $\theta^*$ for this algebra automorphism.  Suppose that  $\{e_1, e_2, \cdots, e_n\}$ is a $\mathbbm{k}$-linear basis of $V$  and $\{e_1^*, e_2^*, \cdots, e_n^*\}$ is the corresponding dual basis of $V^*$. If $\theta(e_i)=\sum\limits_jc_{ij}e_j$ for $c_{ij}
\in\mathbbm{k} \, (1 \leq i, j \leq n)$,
 then we have:
\begin{equation} \label{eqdual} \theta^*(e_i^*)=\sum_jc_{ji}e_j^*.
\end{equation} Moreover, for each $i, j=1, 2, \cdots, n$, we have:
\begin{equation} \label{eqdual1} \theta^*(e_i^*)(e_j)=e_i^*(\theta(e_j)) .
\end{equation}

\begin{proposition}\cite[Theorem 9.2]{VdB97}\label{criteria} Let $A$ be a Koszul
AS-regular algebra of dimension $d$.  Then, the Nakayama automorphism $\nu$ of $A$ is equal to $\epsilon^{d+1}\mu^{*}$,  where $\mu$ is the Nakayama automorphism of the Frobenius algebra $A^!$ and $\epsilon$ is the automorphism of $A$ defined by  $a \mapsto (-1)^{\deg a}a$, for each homogeneous element $a\in A$. \qed
\end{proposition}

Next, we recall the definition and some basic properties
of a double Ore extension.
\begin{definition} \label{def}\cite{ZZ08, ZZ09} Let $A$ be a subalgebra of a $\mathbbm{k}$-algebra $B$.

 $(1)$.  $B$ is called a right double Ore extension of $A$ if
\begin{enumerate}
\item[(\rmnum{1})] $B$ is generated by $A$ together with two variables $y_1$ and $y_2$;
\item[(\rmnum{2})] $y_1$ and $y_2$ satisfy the relation:
\begin{equation*} \label{newer} y_2y_1=py_1y_2+qy_1^2+\tau_1y_1+\tau_2y_2+\tau_0
\end{equation*}
for some $p, q\in \mathbbm{k}$ and $\tau_1, \tau_2, \tau_0\in A$;
\item[(\rmnum{3})] $B$ is a free left $A$-module with basis $\{y_1^iy_2^j;i, j\geq 0\}$;

\item[(\rmnum{4})] $y_1A+y_2A\subseteq Ay_1+Ay_2+A$.
\end{enumerate}
$(2)$.  $B$ is called a left double Ore extension of $A$ if:
\begin{enumerate}
\item[(\rmnum{1})] $B$ is generated by $A$ and two new variables $y_1$ and $y_2$;
\item[(\rmnum{2})] $y_1$ and $y_2$ satisfy the relation
\begin{equation*} %\label{newer}
y_1 y_2=p^\prime y_2y_1+q^\prime y_1^2+ y_1\tau_1^\prime+ y_2\tau_2^\prime+\tau_0^\prime
\end{equation*}
for some $p^\prime, q^\prime\in \mathbbm{k}$ and $\tau_1^\prime, \tau_2^\prime, \tau_0^\prime\in A$;
\item[(\rmnum{3})] $B$ is a free right $A$-module with basis $\{y_2^iy_1^j;i, j\geq 0\}$;
\item[(\rmnum{4})] $Ay_1+Ay_2\subseteq y_1A+y_2A+A$.
\end{enumerate}
$(3)$.  $B$ is called a double Ore extension of $A$ if it is a left and a right double Ore extension of $A$ with the same
generating set $\{y_1, y_2\}$.
\end{definition}

Note that Condition (1).(\rmnum{4}) in the Definition 1.5 is equivalent to the existence of two maps:
$$
\sigma=\begin{pmatrix} \sigma_{11} & \sigma_{12} \\ \sigma_{21} & \sigma_{22}
\end{pmatrix}: A\rightarrow M_{2\times 2}(A)\quad\quad  \text{and}\quad\quad \delta=\begin{pmatrix}\delta_1 \\  \delta_2
\end{pmatrix}: A\rightarrow M_{2\times 1}(A)
$$
subject to
\begin{equation}  \label{sigma} \begin{pmatrix}y_1 \\  y_2
\end{pmatrix}a=\sigma(a)\begin{pmatrix}y_1 \\  y_2
\end{pmatrix}+\delta(a)
\end{equation}
for all $a\in A$, where $\sigma_{ij}, \delta_i\in \mathrm{End}_{\mathbbm{k}}(A)$. {\it In case $B$ is a right double Ore extension of $A$, we will write $B=A_P[y_1, y_2; \sigma, \delta, \tau]$, where $P=(p, q)\in \mathbbm{k}^2$,  $\tau=(\tau_0, \tau_1, \tau_2)\in A^3$,  and $\sigma, \delta$ as above}.
Along with the datum $P$, we define a matrix $\mathbb{P}$ in $M_2(\mathbbm{k})$  as follows:
\begin{equation} \label{matrixp}
\mathbb{P}=\begin{pmatrix} p& 0 \\ -(1+\frac{1}{p})q& \frac{1}{p}
\end{pmatrix}\end{equation}
Like in an Ore extension, here $\sigma$ is a homomorphism of algebras
and $\delta$ is a $\sigma$-derivation, that is, $\delta$ is $\mathbbm{k}$-linear and satisfies $\delta(ab)=\sigma(a)\delta(b)+\delta(a)b$, for all $a, b\in A$.
The double Ore extensions that we shall consider mainly in this work are the so-called {\it trimmed} double Ore extensions.

\begin{definition}\label{dfn1.6}
 A double Ore extension $A_P[y_1, y_2; \sigma, \delta, \tau]$ is called a {\it trimmed  double Ore extension},  if $\delta$ is the zero map and $\tau=\{0, 0, 0\}$. In this case, we use the short notation $A_P[y_1, y_2; \sigma]$ for a trimmed double Ore extension.
\end{definition}

Condition (2).(\rmnum{4}) in Definition \ref{dfn1.6} is equivalent to the existence of two maps
$$
\phi=\begin{pmatrix} \phi_{11} & \phi_{12} \\ \phi_{21} & \phi_{22}
\end{pmatrix}: A\rightarrow M_{2\times 2}(A)\quad \text{and}\quad \delta^\prime=\begin{pmatrix}\delta_1^\prime &  \delta_2^\prime
\end{pmatrix}: A\rightarrow M_{1\times 2}(A)
$$
satisfying
\begin{equation} \label{phi} a\begin{pmatrix}y_1 & y_2
\end{pmatrix}=\begin{pmatrix}y_1 &  y_2
\end{pmatrix}\phi(a)+\delta^\prime(a)
\end{equation}
for all $a\in A$. %And, in this case, we will write $B=A_{P^\prime}[y_1, y_2; \phi, \delta^\prime, \tau^\prime]$, where $P^\prime=\{p^\prime, q^\prime\}\subset k$,  $\tau^\prime=\{\tau_0^\prime, \tau_1^\prime, \tau_2^\prime\}\subset A$,  and $\phi, \delta^\prime$ as above.
For a double Ore extension, the connection between $\sigma$ and $\phi$ in Equation \eqref{sigma} and Equation \eqref{phi} can  be seen in the following definition and lemma.

\begin{definition} \label{inverse}\cite{ZZ08} Let $\sigma:  A\rightarrow M_{2\times 2}(A)$ be an algebra homomorphism. We say that
$\sigma$ is invertible if there is an algebra homomorphism
$\phi=\begin{pmatrix} \phi_{11} & \phi_{12} \\ \phi_{21} & \phi_{22}
\end{pmatrix}: A\rightarrow M_{2\times 2}(A)$ satisfies the following conditions:\\
$$\sum\limits_{k=1}^ 2\phi_{jk}(\sigma_{ik}(r))=\begin{cases}r, & i=j   \\
 0, & i\neq j
\end{cases} \quad\text{and}\quad \sum\limits_{k=1}^2\sigma_{kj}(\phi_{ki}(r))=\begin{cases}r, & i=j   \\
 0, & i\neq j
\end{cases}$$
for all $r\in A$. The map $\phi$ is called the inverse of $\sigma$.
\end{definition}

The following lemma gives the relation of the condition that $\sigma$ is invertible and the condition a right double Ore extension being a double Ore extension.

\begin{lemma} \cite[Lemma 1.9 and its proof, Proposition 2.1]{ZZ08} \label{doeinv} Let $B=A_P[y_1, y_2; \sigma, \delta, \tau]$ be a right double Ore extension of $A$.
\begin{enumerate}
\item[(1)] If $B$ is a double Ore extension, then $\sigma$ is invertible with the inverse $\phi$ such that the Equation \eqref{phi} holds for some $\delta^\prime$.

\item[(2)] Suppose that both $A$ and $B$ are  connected graded algebras. If $p\neq 0$ and $\sigma$ is invertible, then
$B$ is a double Ore extension.\qed \end{enumerate}
\end{lemma}

Next, we list the identities induced by commuting the equation $y_2y_1=py_1y_2+qy_1^2$ with element $r\in A$.  Explicitly,   since $y_2y_1r=(py_1y_2+qy_1^2)r$ for each $r\in A$, so we get the relations R$3.1$-R$3.3$ in \cite[p. 2674]{ZZ08} (as we only consider the trimmed double Ore extension here). Dually,  we have the following
\begin{enumerate}
\item [(R$^\prime$3.1)] $\phi_{11}(\phi_{12}(r))+q\phi_{11}(\phi_{22}(r))=p\phi_{12}(\phi_{11}(r))+q\phi_{11}(\phi_{11}(r))
    +pq\phi_{12}(\phi_{21}(r))+q^2\phi_{11}(\phi_{21}(r))$

\item [(R$^\prime$3.2)] $\phi_{21}(\phi_{12}(r))+p\phi_{11}(\phi_{22}(r))=p\phi_{22}(\phi_{11}(r))+q\phi_{21}(\phi_{11}(r))
    +p^2\phi_{12}(\phi_{21}(r))+pq\phi_{11}(\phi_{21}(r))$

\item [(R$^\prime$3.3)] $\phi_{21}(\phi_{22}(r))=p\phi_{22}(\phi_{21}(r))+q\phi_{21}(\phi_{21}(r))$
\end{enumerate}

In order to study the regularity of double Ore extensions, Zhang and Zhang introduced an invariant of
$\sigma$, called the (right) determinant of $\sigma$, which is similar to the quantum  determinant of the $2\times 2$-matrix.
As we will see, this invariant will play an important role in the description of the Nakayama automorphism of the trimmed double Ore extension.

Let $B=A_P[y_1, y_2; \sigma, \delta, \tau]$ be a right double Ore extension of $A$.
The right determinant of $\sigma$ is defined to be the map:
\begin{equation} \label{detdef}
\det_r\sigma:\ A\longrightarrow A, \  a\mapsto -q\sigma_{12}(\sigma_{11}(a))+\sigma_{22}(\sigma_{11}(a))-p\sigma_{12}(\sigma_{21}(a))
\end{equation}
for  $a\in A$. If $\sigma$ is invertible with the inverse $\phi$, then the left determinant of $\phi$ is defined by: $$\det_l\phi:=-q\phi_{11}\circ\phi_{21}+\phi_{11}\circ\phi_{22}-p\phi_{12}\circ\phi_{21}.$$

We remark  that  when $q=0$ the above expression of $\det_l\phi$ coincides with the one in [ZZ08] after E2.1.6.
The following properties of the determinant of $\sigma$ were given in \cite{ZZ08}.
\begin{proposition}\label{determ}  \cite[proof of Proposition 2.1]{ZZ08} Let $B=A_P[y_1, y_2; \sigma, \delta, \tau]$ be a double Ore extension of $A$
such that $\sigma$ is invertible with inverse $\phi$. Then,
\begin{enumerate}
\item[(1)] $\det_r\sigma$ is an algebra endomorphism of $A$;
\item[(2)] if $p\neq 0$, then \begin{align*}\det_r\sigma&=\frac{q}{p}\sigma_{11}\circ\sigma_{12}+\sigma_{11}\circ\sigma_{22}-\frac{1}{p}\sigma_{21}\circ\sigma_{12},\\ \det_l\phi&=\frac{q}{p}\phi_{21}\circ\phi_{11}+\phi_{22}\circ\phi_{11}-\frac{1}{p}\phi_{21}\circ\phi_{12};
    \end{align*}%(correct the sign in \cite{ZZ08}: $-\frac{q}{p}\to \frac{q}{p}$)
 \item[(3)]   $\det_r\sigma$ is invertible with inverse $\det_l\phi$.
\end{enumerate}
\end{proposition}

Remark that  the equation $\det_l\phi=\frac{q}{p}\phi_{21}\circ\phi_{11}+\phi_{22}\circ\phi_{11}-\frac{1}{p}\phi_{21}\circ\phi_{12}$
follows from the relation R$^\prime3.2$.
Here, we use the notions of $\det_r$ and $\det_l$ in order to differ the morphisms determined by $\sigma$ and $\phi$. Note that there is a print typos in the formula of \cite[line -11, page 2677]{ZZ08}, where the minus sign of the first term should be dropped.  In fact, that can be verified by using \cite[R3.2, page 2674]{ZZ08}.

Double Ore extensions are used to construct higher dimensional AS-regular algebras from  lower dimensional ones because of the following result which  will be used later.

\begin{lemma}\cite[Theorem 0.2]{ZZ08} Let $A$ be an AS-regular algebra. If $B$ is a connected graded  and a double Ore extension of $A$, then $B$ is AS-regular and $\gldim B=\gldim A+2.$ \qed
\end{lemma}

\section{Koszul algebra and homological determinant}
In this section, we make necessary preparation for computing the Nakayama automorphism of a trimmed double Ore extension. To this aim, we first prove  that the Koszul property is preserved by making a trimmed double Ore extension. We then introduce the  homological determinant of an algebra homomorphism $\sigma: A\to M_{2\times 2}(A)$ for a Koszul algebra $A$ and study its properties.

\begin{theorem} \label{koszulity} Let $A$ be a Koszul algebra and  $B=A_P[y_1, y_2; \sigma]$ be a trimmed double Ore extension of $A$. Then, $B$ is a Koszul algebra.
\end{theorem}
\begin{proof} Suppose that $M$ is a $B$-$A$-bimodule and $\varphi$ is an
automorphism of $A$. Recall that $^1M^\varphi$ is the twisted bimodule on the $\mathbbm{k}$-space $M$ with
$$b\cdot m\cdot a=bm\varphi(a)$$
for all $m\in M, b\in B$ and $a\in A$. On the space $M\oplus M$, there is another right $A$-module structure defined by using $\sigma$ as follows:
\begin{equation}\label{mds}
(m, n)\circ a=(m, n)\begin{pmatrix} \sigma_{11}(a) & \sigma_{12}(a) \\ \sigma_{21}(a) & \sigma_{22}(a)
\end{pmatrix}=(m\sigma_{11}(a)+n\sigma_{21}(a), m\sigma_{12}(a)+n\sigma_{22}(a))\end{equation}
for all $m, n\in M$ and $a\in A$.
Since $\sigma$ is an algebra homomorphism, $M\oplus M$ is a $B$-$A$-bimodule. Denote by $(M\oplus M)^{\sigma}$ this $B$-$A$-bimodule.
By \cite[Theorem 2.2]{ZZ08}, there is an exact
sequence of $B$-$A$-bimodules
\begin{equation} \label{ab} 0 \to B^{\det_r\sigma} \stackrel{g}{\to} (B\oplus B)^{\sigma}\stackrel{f}{\to} B \stackrel{\varepsilon}{\to}
A  \to 0,
\end{equation}
where, $f$ maps $(s, t)$ to $sy_1+ty_2$,  $g$ sends $r$ to $(r(qy_1-y_2), rpy_1)$ and the last term $A$ is identified with $B/(y_1, y_2)$. Moreover, \eqref{ab} is a linear resolution of $_BA$ in case  both $y_1$ and $y_2$ are of degree $1$.

Now,  by assumption, $_A\mathbbm{k}$ admits a projective resolution:
\begin{equation} \label{kosa}\cdots{\longrightarrow} P_n{\longrightarrow} P_{n-1} {\longrightarrow}
\cdots \longrightarrow P_1{\longrightarrow}
P_0{\longrightarrow}_A\mathbbm{k}{\longrightarrow}0
\end{equation} with $P_n$
generated in degree $n$ for each $n \geq 0$. We consider the third quadrant bicomplex:
$$\xymatrix{  & 0
 & 0  & 0   &  \\
\cdots \ar[r] & B\otimes_AP_2 \ar[u]
\ar[r] & B\otimes_AP_1 \ar[u]
\ar[r]&  B\otimes_AP_0\ar[r] \ar[u] &0 \\
\cdots \ar[r] &(B\oplus B)^\sigma\otimes_AP_2\ar[u] \ar[r]& (B\oplus B)^\sigma\otimes_AP_1 \ar[u]
\ar[r]& (B\oplus B)^\sigma\otimes_AP_0\ar[r] \ar[u]  &  0 \\
\cdots  \ar[r] &B^{\det_r \sigma}\otimes_AP_2
\ar[r]\ar[u] & B^{\det_r \sigma}\otimes_AP_1 \ar[u]
\ar[r]&  B^{\det_r \sigma}\otimes_AP_0\ar[r] \ar[u] &0\\
&0 \ar[u] & 0\ar[u]&  0 \ar[u] & }$$

It follows that $\det_r \sigma$ is an automorphism of $A$ and that $B$ is a right free $A$-module, $B^{\det_r \sigma}$ is projective as a right $A$-module.
Now, for the right $A$-module $(B\oplus B)^\sigma$, we are going to show it is also projective as a right $A$-module.
Since we have the following general result: if $M, P, Q$ are projective in the exact sequence
\begin{equation*}  0 \to M \to N\to P\to
Q  \to 0,
\end{equation*}
then so is $N$. For this end, we take $K$ to be the kernel of $P\to
Q$ and we get two short exact sequence
\begin{equation*}  0 \to M \to N\to K\to 0,\quad \quad \quad 0\to K\to P\to
Q  \to 0.
\end{equation*}
Since $P$ and $Q$ are projective, the second sequence is split and $K$ is projective. Therefore,  the first sequence is split and $N$ is projective.
Hence,  each term in the sequence \eqref{ab} is projective as a right $A$-module.
Further, all the rows of the bicomplex are exact except at the $(-1)$-st column.
Thus, the homology along the rows  yields a single nonzero column, that is,
\begin{equation} \label{abcom} \cdots\to 0 \to B^{\det_r\sigma}\otimes_A\mathbbm{k} {\to} (B\oplus B)^{\sigma}\otimes_A\mathbbm{k}{\to} B \otimes_A\mathbbm{k} \to 0.
\end{equation}
Moreover,  the sequence \eqref{ab} is a
split exact sequence. Therefore, the homology of \eqref{abcom} is $_BA\otimes_A\mathbbm{k}=_B\mathbbm{k}$.
Namely, the total complex of the bicomplex is a projective resolution of the $B$-module $_B\mathbbm{k}$.
Finally, both sequence \eqref{ab}  and \eqref{kosa} are linear resolutions, so is the total complex of the bicomplex. The proof is completed.
\end{proof}

\begin{remark}
(1). Theorem \ref{koszulity} generalizes  the well-known result that a graded
Ore extension of a Koszul algebra is again Koszul (see \cite[Corollary 1.3]{Ph12}).

(2). It was proved  in \cite[Theorem 0.1(b)]{ZZ09} that a graded double Ore extension of an AS-regular algebra of dimension 2 is Koszul. Since  an AS-regular algebra of dimension 2 is always Koszul, Theorem \ref{koszulity} generalizes \cite[Theorem 0.1(b)]{ZZ09} in the trimmed case.
\end{remark}

For  an AS-Gorenstein algebra $A$, J{\o}rgensen and Zhang proposed the notion of the homological determinant of a
graded automorphism in \cite{JZ00} in order to study the noncommutative invariant theory. Roughly speaking,  for an AS-Gorenstein
algebra $A$, the homological determinant, denoted $hdet$, is a homomorphism from  the graded automorphism group $\GrAut(A)$ of $A$ to the multiplicative group  $\mathbbm{k}\backslash \{0\}$ generalizing the usual determinant of a matrix. For the precise definition and its application, we refer to \cite{JZ00, RRZ14}. Here, we just need the following characterization of the homological determinant of an automorphism of a Koszul algebra.

\begin{proposition} \cite[Proposition 1.11]{WZ11} \label{hdk}
Let $A$ be a Koszul AS-regular algebra of global
dimension $d$. Suppose that $\theta$ is a graded automorphism of $A$ and
$\theta^*$ is its corresponding dual graded automorphism of the dual algebra $A^!$. Then, we have
$\theta^*(u)=(\hdet\theta)u$ for any $u\in\Ext_A^d(\mathbbm{k}, \mathbbm{k})$.\qed
\end{proposition}

Suppose that $A=T_\mathbbm{k}(V)/\langle R\rangle$ is a Koszul algebra. Let $\sigma:  A\rightarrow M_{2\times 2}(A)$ be an algebra homomorphism.
Then, $$\begin{pmatrix} \sigma_{11}^* & \sigma_{12}^* \\ \sigma_{21}^* & \sigma_{22}^*
\end{pmatrix}: V^*\rightarrow M_{2\times 2}(V^*)$$ defines a $\mathbbm{k}$-linear map, denoted by $\sigma^*$, where $\sigma_{ij}^*$ is the dual of $\sigma_{ij}$ on the space $V^*$ (see the paragraph before Proposition \ref{criteria}) for each pair $(i, j)$ with $i, j\in \{1, 2\}$.
Extend $\sigma^*$ to an algebra homomorphism $\sigma^*: T_\mathbbm{k}(V^*)\rightarrow M_{2\times 2}(T_\mathbbm{k}(V^*))$ by letting: $$\sigma^*(xy):=\sigma^*(x)\sigma^*(y)$$ for each $x, y\in V^*$.
In particular, for $e_i^*, e_j^*\in V^*$
\begin{align*} \sigma^*(e_i^*e_j^*)=&\begin{pmatrix} \sigma_{11}^*(e_i^*) & \sigma_{12}^*(e_i^*) \\ \sigma_{21}^*(e_i^*) & \sigma_{22}^*(e_i^*)
\end{pmatrix}\begin{pmatrix} \sigma_{11}^*(e_j^*) & \sigma_{12}^* (e_j^*)\\ \sigma_{21}^* (e_j^*)& \sigma_{22}^*(e_j^*)
\end{pmatrix}\\
=&\begin{pmatrix} \sigma_{11}^*(e_i^*)  \sigma_{11}^*(e_j^*) +\sigma_{12}^*(e_i^*)  \sigma_{21}^*(e_j^*)&
\sigma_{11}^*(e_i^*)  \sigma_{12}^*(e_j^*) +\sigma_{12}^*(e_i^*)  \sigma_{22}^*(e_j^*)\\
\sigma_{21}^*(e_i^*)  \sigma_{11}^*(e_j^*) +\sigma_{22}^*(e_i^*)  \sigma_{21}^*(e_j^*)&
\sigma_{21}^*(e_i^*)  \sigma_{12}^*(e_j^*) +\sigma_{22}^*(e_i^*)  \sigma_{22}^*(e_j^*)
\end{pmatrix}.
\end{align*}
For any $e_k, e_l\in V$,
\begin{align*}(\sigma_{11}^*(e_i^*)&  \sigma_{11}^*(e_j^*) +\sigma_{12}^*(e_i^*)  \sigma_{21}^*(e_j^*))(e_ke_l)\\
&\overset{\text{by} \eqref{rperp}}=
\sigma_{11}^*(e_i^*)(e_k)  \sigma_{11}^*(e_j^*)(e_l)+\sigma_{12}^*(e_i^*)(e_k) \sigma_{21}^*(e_j^*)(e_l)\\
&\overset{\text{by} \eqref{eqdual1}}=e_i^*(\sigma_{11}(e_k))e_j^*(\sigma_{11}(e_l))+e_i^*(\sigma_{12}(e_k))e_j^*(\sigma_{21}(e_l))\\
&\overset{\text{by} \eqref{rperp}}=e_i^*e_j^*(\sigma_{11}(e_k)\sigma_{11}(e_l)+\sigma_{12}(e_k)\sigma_{21}(e_l))\\
&=e_i^*e_j^*\big((\sigma_{11}\sigma_{11}+\sigma_{12}\sigma_{21})(e_ke_l)\big).
\end{align*}
Then,  $\sigma_{11}^*(r^\prime)\in R^\perp$ for any $r^\prime\in R^\perp$. For this end,
assuming that $r^\prime=\sum\limits_{i, j}c_{ij}e_i^*e_j^*\in R^\perp$, then for any  $r=\sum\limits_{k, l}d_{kl}e_ke_l\in R$ it follows from the above computation that
\begin{align*} \sigma_{11}^*(r^\prime)(r)&=\sum\limits_{i, j}c_{ij}(\sigma_{11}^*(e_i^*)\sigma_{11}^*(e_j^*) +\sigma_{12}^*(e_i^*)  \sigma_{21}^*(e_j^*))(\sum\limits_{k, l}d_{kl}e_ke_l)\\
&=\sum\limits_{i, j}c_{ij}e_i^*e_j^*\big((\sigma_{11}\sigma_{11}+\sigma_{12}\sigma_{21})(\sum\limits_{k, l}d_{kl}e_ke_l)\big)\\
&=r^\prime(\sigma_{11}(r)).
\end{align*}
Since $\sigma_{11}$ is an algebra endomorphism of $A=T_\mathbbm{k}(V)/\langle R\rangle$,  we obtain that $\sigma_{11}(r)\in R$. Hence, $\sigma_{11}^*(r^\prime)(r)=r^\prime(\sigma_{11}(r))=0$.  It is shown  that $\sigma_{11}^*(r^\prime)\in R^\perp$ for any $r^\prime\in R^\perp$.
That is, $\sigma_{11}^*$ induces an algebra endomorphism of $A^!=T_\mathbbm{k}(V^*)/\langle R^\perp\rangle$. Similarly, the same claims for $\sigma_{12}^*$, $\sigma_{21}^*$ and $\sigma_{22}^*$ hold by computation.
Furthermore,  $\sigma^*$ induces an
algebra homomorphism from  $A^!$ to $M_{2\times 2}(A^!)$. We still use the same notation $\sigma^*$ for this algebra homomorphism if no confusion occurs.  The following property is easy to check.

%This construction is also valid for the algebra morphism $\phi=\begin{pmatrix} \phi_{11} & \phi_{12} \\ \phi_{21} & \phi_{22}
%\end{pmatrix}: A\rightarrow M_{2\times 2}(A)$.
%Moreover, it is similar to the relation between the automorphism group of a Koszul algebra and the one of its Koszul dual that

\begin{lemma} \label{invsp}
Let $A$ be a Koszul algebra and  $\sigma: A\rightarrow M_{2\times 2}(A)$  an algebra homomorphism. Then $\sigma$ is invertible ( in the sense of Definition \ref{inverse})  with  inverse $\phi$ if and only if $\sigma^*$ is invertible with  inverse $\phi^*$. Here both $\sigma^*$ and $\phi^*$ are algebra homomorphisms from $A^!$ to $M_{2\times 2}(A^!)$.
\end{lemma}

 Let  $x_0$ be  a base element of the highest nonzero component $A^!_d$, which is $1$-dimensional $\mathbbm{k}$-space, of $A^!$.  We assume that:
\begin{equation} \label{hdet}
\sigma^*(x_0)=\begin{pmatrix} Wx_0 & X x_0 \\ Yx_0 & Zx_0
\end{pmatrix}, \quad \phi^*(x_0)=\begin{pmatrix} W^\prime x_0 & X^\prime x_0 \\ Y^\prime x_0 & Z^\prime x_0
\end{pmatrix}
\end{equation} for some $W, X, Y, Z, W^\prime, X^\prime, Y^\prime, Z^\prime\in \mathbbm{k}$.

Inspired by  Proposition \ref{hdk}, we may introduce the following:

\begin{definition}  \label{hdett} Let $A$ be a Koszul AS-regular algebra. Suppose that $\sigma$ is an algebra homomorphism from $A$ to $M_{2\times 2}(A)$  and
$\sigma^*$ is its dual algebra homomorphism from $A^!$ to $M_{2\times 2}(A^!)$. The homological determinant of $\sigma$, denoted $\hdet\sigma$,  is defined by
$$\hdet\sigma:=\begin{pmatrix} W& X \\ Y& Z
\end{pmatrix},$$
where $W,X,Y$ and $Z$ are determined by (\ref{hdet}).
\end{definition}

%
%to be the matrix $\begin{pmatrix} W& X \\ Y& Z
%\end{pmatrix}$, where this matrix is defined by equation \eqref{hdet}.
% Then, we have
%$\theta^*(u)=(\hdet\theta)u$ for any $u\in\Ext_A^d(\mathbbm{k}, \mathbbm{k})$.\qed
The following property follows directly from Lemma \ref{invsp}.
\begin{lemma} \label{matrixin}
Let $A$ be a Koszul algebra and  $\sigma$  and $\phi$ be two algebra homomorphism from $A$ to $M_{2\times 2}(A)$ such that they are inverse of each other  in the sense of Definition \ref{inverse}. Then
$$\hdet\sigma(\hdet\phi)^t=I_2,$$
or equivalently,
$$\begin{pmatrix} W& X \\ Y& Z
\end{pmatrix}\begin{pmatrix} W^\prime & Y^\prime \\ X^\prime & Z^\prime
\end{pmatrix}=I_2 ,$$ where $M^t$ is the transpose of a matrix $M$ and $I_2$ is the $2\times 2$ identity matrix.\qed
\end{lemma}

\begin{example} \label{exhdet} Let $A$ be a Koszul AS-regular algebra and  $B=A_P[y_1, y_2; \sigma]$ be a trimmed double Ore extension of $A$ with $
\sigma=\begin{pmatrix} \tau & 0 \\ 0 & \xi
\end{pmatrix}$.
Then, both $\tau$ and $\xi$ are automorphisms of $A$ and $\tau\xi=\xi\tau$ (see Proposition \ref{iterated} for its proof). Moreover, $B$ is an iterated Ore extension of $A$ by \cite[Theorem 2.2]{CLM11}.
It is easy to see that
$$\hdet\sigma=\begin{pmatrix} \hdet\tau & 0 \\ 0 & \hdet\xi
\end{pmatrix}.$$
\end{example}

\section{Nakayama automorophisms}
In this section, we study the Yoneda Ext algebra of a trimmed double Ore extension of a Koszul AS-regular algebra,  and  compute the Nakayama automorphism of the trimmed double Ore extension. This leads to  the characterization of the Calabi-Yau property of a trimmed double Ore extension. As consequences,  we recover several known results on the Calabi-Yau property of a skew polynomial extension.

Throughout this section, $A=T_\mathbbm{k}(V)/\langle \,R\, \rangle$ is a Koszul AS-regular algebra of global dimension
$d$ with Nakayama automorphism $\nu$,  and $B=A_P[y_1, y_2; \sigma]$ is a trimmed double Ore extension of $A$,  where $\sigma=\begin{pmatrix} \sigma_{11} & \sigma_{12} \\ \sigma_{21} & \sigma_{22}
\end{pmatrix}$ is an algebra morphism subject to  (\ref{sigma}).  Let $\phi=\begin{pmatrix} \phi_{11} & \phi_{12} \\ \phi_{21} & \phi_{22}
\end{pmatrix}$  be the inverse of $\sigma$ in the sense of  (\ref{inverse}),
$\hdet\sigma=\begin{pmatrix} W& X \\ Y& Z
\end{pmatrix},$  and $\hdet\phi=\begin{pmatrix} W'& X' \\ Y'& Z'
\end{pmatrix}$ throughout this section. We choose
a basis  $\{e_1,  \cdots, e_n\}$  of $V$,  and let $\{e_1^*,  \cdots, e_n^*\}$ be  the corresponding dual basis of $V^*$.
For the the Frobenius algebra $A^!$, we fix  a base element $x_0$ of  the $1$-dimensional $\mathbbm{k}$-space $A_d^!$. By \cite[Lemma 3.2]{Sm96},   $A^!$ possesses a nondegenerate bilinear form given by
\begin{equation} \label{bilinear} \langle a, b\rangle=c_{ab}
\end{equation}
where $c_{ab}$ is the coefficient of $x_0$ in the product of $ab$.
We can pick a $\mathbbm{k}$-linear basis $\{\eta_1, \eta_2, \cdots, \eta_n\}$
of $A_{d-1}^!$ such that $e_i^*\eta_j=\delta_{ij} x_0$. Then
$\eta_ie_j^*=\lambda_{ij} x_0$ for some  $\lambda_{ij}\in \mathbbm{k}$. Or equivalently,
$$\langle e_i^*, \eta_j\rangle=\delta_{ij}, \quad \langle\eta_i, e_j^*\rangle=\lambda_{ij}$$ for $i,j=1, 2, \cdots,n$.
Then, it follows from \eqref{nakdef} that the Nakayama automorphism $\mu_{A^!}$ of $A^!$ is given by:
\begin{equation} \label{nakaa} \mu_{A^!}(e_i^*)=\sum\limits_{j} \lambda_{ji}e_j^*.
\end{equation}
Now  we  assume that the algebra homomorphism $\phi=\begin{pmatrix} \phi_{11} & \phi_{12} \\ \phi_{21} & \phi_{22}
\end{pmatrix}: A\rightarrow M_{2\times 2}(A)$  is given by
\begin{equation} \label{hpiij}\phi_{ij}(e_l)=\sum_k \phi_{ij}^{lk}e_k
\end{equation}   for each $l$, where  $\phi^{lk}_{ij}\in \mathbbm{k}$. Then, we have
\begin{equation} \label{hpistij}\phi_{ij}^*(e_l^*)=\sum_k \phi_{ij}^{kl}e_k^*.
\end{equation}

Now  $B$ is a Koszul algebra and it can be presented by generators and relations as  $B=T_\mathbbm{k}(V \oplus \mathbbm{k}y_1\oplus \mathbbm{k}y_2)/\langle \, R_B \rangle$, where  $R_B$ consists of
three  types of relations:
\begin{enumerate}
\item [(R1)] the relations defining $A$;

\item [(R2)] $y_2y_1-py_1y_2-qy_1^2$;

\item [(R3)] $\{y_je_i-\sigma_{j1}(e_i)y_1-\sigma_{j2}(e_i)y_2; j=1, 2, i=1, \cdots, n\}.$

\end{enumerate}

 Note that from Definition \ref{def} and Definition \ref{inverse} it follows that the relation $($R$3)$  is equivalent to
\begin{enumerate}

\item [(R3')] $\{e_i y_j-y_1\phi_{1j}(e_i)-y_2\phi_{2j}(e_i); j=1, 2, i=1, \cdots, n\}.$

\end{enumerate}

Let $C:=\mathbbm{k}\langle y_1, y_2\rangle/\langle y_2y_1-py_1y_2-qy_1^2\rangle(p\neq 0)$. We need  the following well-known property of  the algebra $C$.

\begin{proposition} \label{dext}  The algebra $C$ is Koszul AS-regular of dimension $2$. Its Yoneda Ext algebra $C^!$ is
$\mathbbm{k}\langle y_1^*, y_2^*\rangle/\langle(y_1^*)^2+qy_2^*y_1^*, y_1^*y_2^*+py_2^*y_1^*, (y_2^*)^2\rangle$.
\end{proposition}
\begin{proof} The algebra is known as the Jordan plane $(q\neq 0)$ or quantum plane $(q= 0)$ which are both Koszul AS-regular of dimension $2$.
Its Yoneda Ext algebra $E(C):=\bigoplus_{i\in \mathbb{N}}\Ext_C^i(\mathbbm{k}, \mathbbm{k})$ is isomorphic to $C^!=T_\mathbbm{k}(V^*)/\langle \, R^\perp \,\rangle$, see \cite[Theorem 5.9]{Sm96}.
\end{proof}
Next, we can describe the algebra $B^!$ in terms of generators and relations. It is obvious that  $\{e_1^*, e_2^*, \cdots, e_n^*, y_1^*, y_2^*\}$  forms a $\mathbbm{k}$-linear basis of $B_1^!$.

\begin{lemma} \label{lem5a}   The algebra  $B^!$ is generated by elements $\{e_1^*,
e_2^*, \cdots, e_n^*, y_1^*, y_2^*\}$ with the relations:
\begin{enumerate}
\item [($\perp$1)] the relations for $A^!$;
\item [($\perp$2)] the relations for $C^!$;
\item [($\perp$3)] $\{y_j^*e_i^*+\phi_{j1}^*(e_i^*)y_1^*+\phi_{j2}^*(e_i^*)y_2^*; j=1, 2, i=1, \cdots, n\},$ where
$\phi$ is the inverse of $\sigma$.
\end{enumerate}
%\item [(2)] For $k\geq 2$,
%  \begin{align*}
%  B^!_k=&(V^*)^k+y_1^*(V^*)^{k-1}+y_2^*(V^*)^{k-1}+y_1^*y_2^*(V^*)^{k-2}\\
%    =&(V^*)^k+(V^*)^{k-1}y_1^*+(V^*)^{k-1}y_2^*+(V^*)^{k-2}y_1^*y_2^*,
% \end{align*}
%where  $(V^*)^k$ denotes the $k$-th component of the subalgebra of $B^!$ generated by $V^*$.
%\end{enumerate}
\end{lemma}

\begin{proof}  Since $B$ is Koszul,  we have $B^!=T_\mathbbm{k}(V^* \oplus \mathbbm{k}y_1^*\oplus \mathbbm{k}y_2^*)/\langle \, (R_B)^\perp \rangle$.  According to the defining relations of $B$,
it is easy to see that relations ($\perp$1) and ($\perp$2) belong to
$(R_B)^\bot$.  Now  we show that ($\perp$3) also belongs to $(R_B)^\bot$. It suffices to verify that for every $i,j$, we have:
$$\big(y_j^*e_i^*+\phi_{j1}^*(e_i^*)y_1^*+\phi_{j2}^*(e_i^*)y_2^*\big)(r)=0$$
for each $r\in R_B$ by the definition of $(R_B)^\bot$ for each $i, j$. But this is trivial  since the generating relations of $B$ are given by (R$1$), (R$2$) and (R$3$).

On the other hand, each element in $(V^*)^{\otimes 2}$ has the form $f+g+h$, where  $f=\sum_i k_i e_i^*
y_1^* +l_i e_i^*y_2^*+ m_i y_1^*e_i^* + n_i y_2^*e_i^*$, $g=\sum c_{ij}e_i^*e_j^*$ and $h=a(y_1^*)^2+by_1^*y_2^*+cy_2^*y_1^*+d(y_2^*)^2$. Assume that $f+g+h\in (R_B)^\bot$. Then, it is easy to see $g$ is in the span of $(\bot1)$ and $h$ is in the span of $(\bot2)$. For the rest,
we need to show that every element $f=\sum_i k_i e_i^*
y_1^* +l_i e_i^*y_2^*+ m_i y_1^*e_i^* + n_i y_2^*e_i^*\in (R_B)^\bot$ can be written as
$$f=\sum a_i(y_1^*e_i^*+\phi_{11}^*(e_i^*)y_1^*+\phi_{12}^*(e_i^*)y_2^*)+b_i(y_2^*e_i^*+\phi_{21}^*(e_i^*)y_1^*+\phi_{22}^*(e_i^*)y_2^*),$$
for $a_i, b_i\in \mathbbm{k}$. Firstly, we have
$$k_i=\sum_jm_j e_j^*(\phi_{11}(e_i))+n_j e_j^*(\phi_{21}(e_i))$$ and
$$l_i=\sum_jm_j e_j^*(\phi_{12}(e_i))+n_j e_j^*(\phi_{22}(e_i))$$
     for any $i$.
Further,
$$\sum_i e_j^*(\phi_{11}(e_i))e_i^*=\phi_{11}^*(e_j^*)$$
by the definition of $\phi_{11}^*$.  Hence, we have
\begin{align*} f=&\sum_jm_j\phi_{11}^*(e_j^*)y_1^*+n_j\phi_{21}^*(e_j^*)y_1^*\\
&+\sum_jm_j\phi_{12}^*(e_j^*)y_2^*+n_j\phi_{22}^*(e_j^*)y_2^*\\
&+\sum_i  m_i y_1^*e_i^* + n_i y_2^*e_i^*\\
=&\sum_i m_i(y_1^*e_i^*+\phi_{11}^*(e_i^*)y_1^*+\phi_{12}^*(e_i^*)y_2^*)\\
&+n_i(y_2^*e_i^*+\phi_{21}^*(e_i^*)y_1^*+\phi_{22}^*(e_i^*)y_2^*),
\end{align*}
which completes the proof.
\end{proof}

\begin{remark} \label{rem}
The third type of relation $(\bot3)$ of $B^!$ can be replaced by

\begin{enumerate}
\item [($\bot3^\prime$)] $\{e_i^*y_j^*+y_1^*\sigma_{1j}^*(e_i^*)+y_2^*\sigma_{2j}^*(e_i^*); j=1, 2, i=1, \cdots, n\}$
\end{enumerate}
since the relation R$3$ can be replaced by R$3^\prime$.
\end{remark}

\begin{proposition} \label{basis} Suppose that $A$ is a Koszul algebra and  $B=A_P[y_1, y_2; \sigma]$
is a trimmed  double Ore extension of $A$. Then,
\begin{enumerate}
\item [(1)] $A^!$ is a subalgebra of $B^!$;

\item [(2)] %$B^!=A^!\oplus A^!y_1^*\oplus A^!y_2^*\oplus A^!y_1^*y_2^*=A^!\oplus y_1^*A^!\oplus y_2^*A^!\oplus y_1^*y_2^*A^!$. Moreover,
    $B^!$ is a free right (and left) $A^!$-module with a basis $\{1, y_1^*, y_2^*, y_1^*y_2^*\}$.
\end{enumerate}
\end{proposition}

\begin{proof}  The statement (1) is a consequence of Lemma \ref{lem5a}.
Moreover, there is a surjective algebra homomorphism $\pi: A^!\coprod C^!\rightarrow B^!$ from
the coproduct of $A^!$ and  $C^!$ to $B^!$. Hence, as a left $A^!$-module, $B^!$ is generated by $1, y_1^*, y_2^*$ and $y_1^*y_2^*$.
By Lemma \ref{lem5a} and Remark \ref{rem}, the kernel of $\pi$ is the ideal generated by
$$\{e_i^*y_j^*+y_1^*\sigma_{1j}^*(e_i^*)+y_2^*\sigma_{2j}^*(e_i^*); j=1, 2, i=1, \cdots, n\}.$$
Therefore, the elements  $1, y_1^*, y_2^*$ and $y_1^*y_2^*$ are also the generators of $B^!$ as a right $A^!$-module.

Next, since $B$ is a free left $A$-module with basis $\{y_1^iy_2^j;i, j\geq 0\}$ by definition, the Hilbert series of $B$
is equal to the Hilbert series of $A\otimes \mathbbm{k}[y_1, y_2]$, i.e.,
$$\H_B(t)=\frac{\H_A(t)}{(1-t)^2}.$$
It is well known that there is a functional equation on Hilbert series
$$\H_{S}(t)\H_{S^!}(-t)=1$$
for any Koszul algebra $S$. Since both $A$ and $B$ are Koszul algebras by Theorem \ref{koszulity}, so we have
\begin{equation} \label{hilbert}\H_{B^!}(t)=(1+t)^2\H_{A^!}(t).\end{equation}
Therefore,  $B^!$ is a free left(also right) $A^!$-module,
with  a basis $\{1, y_1^*, y_2^*, y_1^*y_2^*\}$.
\end{proof}

%Prove similar prop. to \cite[Proposition 2.5]{LSV}, $B^!$ is a free $A^!$-module with
%basis $\{1, y_1^*, y_2^*\}$. Hence, $\delta y_1^*y_2^*$ is a base element of the
%$1$-dimensional space $B_{d+2}^!$, denoting it by $\varepsilon$.

In order to compute the Nakayama automorphism of $B^!$, we need the following:

%Then,
%\begin{equation} \label{eqdual}\theta^*(e_i^*)=\sum_j b_{ji}e_j^*.
%\end{equation}
%Moreover, we have

\begin{lemma}\label{exts} With notations and assumptions as in the second paragraph of this section, we have

$(1)$  $\varepsilon:=x_0 y_1^*y_2^*$ is a basis element of the
$1$-dimensional space $B_{d+2}^!$.

$(2)$   For any $1\leq i, j\leq n$ and $m=1, 2$, the  following equations hold:

\quad \quad $\begin{cases}
e_i^*\eta_jy_1^*y_2^*\overset{(a)}=\delta_{ij}\varepsilon, & \eta_iy_1^*y_2^*e_j^*\overset{(b)}=\sum\limits_{k, l}(\frac{q}{p}\phi_{21}^{kj}\phi_{11}^{lk}\!-\!\frac{1}{p}\phi_{21}^{kj}\phi_{12}^{lk}
\!+\phi_{22}^{kj}\phi_{11}^{lk})\lambda_{il}\varepsilon,    \\
e_i^*x_0 y_m^*\overset{(c)}=0, & x_0 y_m^* e_i^*\overset{(d)}=0,     \\
y_1^*x_0 y_2^*\overset{(e_1)}=(-1)^dW^\prime\varepsilon, & y_1^*x_0 y_1^*\overset{(e_2)}=(-1)^d(\frac{q}{p}W^\prime-\frac{1}{p}X^\prime)\varepsilon, \\
y_2^*x_0 y_2^*\overset{(e_3)}=(-1)^dY^\prime\varepsilon, &y_2^*x_0 y_1^*\overset{(e_4)}=(-1)^d(\frac{q}{p}Y^\prime-\frac{1}{p}Z^\prime)\varepsilon,\\
x_0 y_1^*y_1^*\overset{(f_1)}=\frac{q}{p}\varepsilon, &x_0 y_1^*y_2^*\overset{(f_2)}=\varepsilon,\\
x_0 y_2^*y_1^*\overset{(f_3)}=-\frac{1}{p}\varepsilon, &x_0 y_2^*y_2^*\overset{(f_4)}=0,\\
 y_m^*\eta_jy_1^*y_2^*\overset{(g)}=0, & \eta_jy_1^*y_2^*y_m^*\overset{(h)}=0,
\end{cases}$\\
 where $W', X',Y'$ and $Z'$ are given in (\ref{hdet}).
\end{lemma}

\begin{proof} Part (1) is obvious by Proposition \ref{basis} (2).
Since $e_i^*\eta_j=\delta_{ij}x_0$, we have
$$e_i^*\eta_jy_1^*y_2^*=\delta_{ij}x_0 y_1^*y_2^*=\delta_{ij}\varepsilon.$$
So Equation (a) holds.  Since
$A^!\rightarrow B^!$ is injective, Equation (c)  holds
naturally.  Equation (d) holds due to the relation ($\bot$3) of Lemma \ref{lem5a}.  Equations (g) and (h) follow from the relations ($\bot$2) and ($\bot$3) of Lemma \ref{lem5a}.
As  for Equation (b), by relation ($\bot$3) of Lemma \ref{lem5a} and Proposition \ref{dext}, we have:
\begin{align*} y_1^*y_2^*e_j^*&=-y_1^*(\phi_{21}^*(e_j^*)y_1^*+\phi_{22}^*(e_j^*)y_2^*)\\
&=-\sum_k(\phi_{21}^{kj}y_1^*e_k^*y_1^*+\phi_{22}^{kj}y_1^*e_k^*y_2^*)\\
&=\sum_k(\phi_{21}^{kj}\phi_{11}^*(e_k^*)y_1^*y_1^*+\phi_{21}^{kj}\phi_{12}^*(e_k^*)y_2^*y_1^*+\phi_{22}^{kj}\phi_{11}^*(e_k^*)y_1^*y_2^*)\\
&=\sum_{k, l}(\phi_{21}^{kj}\phi_{11}^{lk}e_l^*y_1^*y_1^*+\phi_{21}^{kj}\phi_{12}^{lk}e_l^*y_2^*y_1^*
+\phi_{22}^{kj}\phi_{11}^{lk}e_l^*y_1^*y_2^*)\\
&=\sum_{k, l}(\frac{q}{p}\phi_{21}^{kj}\phi_{11}^{lk}-\frac{1}{p}\phi_{21}^{kj}\phi_{12}^{lk}
+\phi_{22}^{kj}\phi_{11}^{lk})e_l^*y_1^*y_2^*.
\end{align*}
Thus, for each $i$
\begin{align*}
\eta_iy_1^*y_2^*e_j^*=&\sum_{k, l}(\frac{q}{p}\phi_{21}^{kj}\phi_{11}^{lk}-\frac{1}{p}\phi_{21}^{kj}\phi_{12}^{lk}
+\phi_{22}^{kj}\phi_{11}^{lk})\eta_ie_l^*y_1^*y_2^*\\
%=&\sum_{k, l}(\frac{q}{p}\phi_{21}^{kj}\phi_{11}^{lk}-\frac{1}{p}\phi_{21}^{kj}\phi_{12}^{lk}
%+\phi_{22}^{kj}\phi_{11}^{lk})\lambda_{il}\delta y_1^*y_2^*\\
=&\sum_{k, l}(\frac{q}{p}\phi_{21}^{kj}\phi_{11}^{lk}-\frac{1}{p}\phi_{21}^{kj}\phi_{12}^{lk}
+\phi_{22}^{kj}\phi_{11}^{lk})\lambda_{il}\varepsilon,
\end{align*} where the second  equation follows from $\eta_ie_l^*=\lambda_{il}x_0$ by the assumption.
Next, we show  the rest equations. For a fixed $j$,  suppose that $\eta_j=\sum\limits_m
\lambda_me^*_{m_1}e^*_{m_2}\cdots e^*_{m_{d-1}}$,  where
$\lambda_m\in \mathbbm{k}$. Then,
\begin{align*} \begin{pmatrix} y_1^* \\  y_2^*
\end{pmatrix}x_0=&\begin{pmatrix} y_1^* \\  y_2^*
\end{pmatrix}e_j^*\eta_j\\
=&\sum\limits_m\lambda_m\begin{pmatrix} y_1^* \\  y_2^*
\end{pmatrix}e_j^*e^*_{m_1}e^*_{m_2}\cdots e^*_{m_{d-1}}\\
=&-\sum\limits_m\lambda_m \phi^*(e_j^*)\begin{pmatrix} y_1^* \\  y_2^*\end{pmatrix}e^*_{m_1}e^*_{m_2}\cdots e^*_{m_{d-1}}\\
=&(-1)^2\sum\limits_m\lambda_m \phi^*(e_j^*)\phi^*(e^*_{m_1})\begin{pmatrix} y_1^* \\  y_2^*\end{pmatrix}e^*_{m_2}\cdots e^*_{m_{d-1}}\\
=&\cdots\\
=&(-1)^d\sum\limits_m\lambda_m \phi^*(e_j^*)\phi^*(e^*_{m_1})\phi^*(e^*_{m_2})\cdots \phi^*(e^*_{m_{d-1}})\begin{pmatrix} y_1^* \\  y_2^*\end{pmatrix}\\
=&(-1)^d\sum\limits_m\lambda_m \phi^*(e_j^*e^*_{m_1}e^*_{m_2}\cdots e^*_{m_{d-1}})\begin{pmatrix} y_1^* \\  y_2^*\end{pmatrix}\\
=&(-1)^d\phi^*(e_j^*\sum\limits_m\lambda_m e^*_{m_1}e^*_{m_2}\cdots e^*_{m_{d-1}})\begin{pmatrix} y_1^* \\  y_2^*\end{pmatrix}\\
=&(-1)^d\phi^*(e_j^*\eta_j)\begin{pmatrix} y_1^* \\  y_2^*\end{pmatrix}
=(-1)^d\phi^*(x_0)\begin{pmatrix} y_1^* \\  y_2^*\end{pmatrix}.
\end{align*}
It follows from  the definition of $\phi^*$ that we obtain:
$$
\begin{pmatrix} y_1^* \\  y_2^*
\end{pmatrix}x_0=(-1)^d\begin{pmatrix} \phi_{11}^*(x_0) & \phi_{12}^*(x_0) \\ \phi_{21}^*(x_0) & \phi_{22}^*(x_0)
\end{pmatrix}\begin{pmatrix}y_1^* \\  y_2^*
\end{pmatrix}=(-1)^d\begin{pmatrix}W^\prime x_0 y_1^*+X^\prime x_0 y_2^* \\  Y^\prime x_0 y_1^*+Z^\prime x_0 y_2^*
\end{pmatrix}.$$
Thus, we have proved  Equations $(e_i), i=1,2,3,4$.  Finally, the equations $(f_i), i=1,\cdots ,4$ and $(g), (h)$ follow from Proposition \ref{dext}.
\end{proof}

Since $B^!$ is Frobenius, we may  apply the Frobenius pair  \eqref{bilinear}  on  the equations in Lemma \ref{exts}(2).

\begin{corollary} \label{value} The following equations hold:
$$\begin{cases}
\langle e_i^*, \eta_jy_1^*y_2^*\rangle\overset{(a^\prime)}= \delta_{ij}, \quad\, &\langle\eta_iy_1^*y_2^*, e_j^*\rangle\overset{(b^\prime)}=\sum\limits_{k, l}(\frac{q}{p}\phi_{21}^{kj}\phi_{11}^{lk}\!-\!\frac{1}{p}\phi_{21}^{kj}\phi_{12}^{lk}
\!+\phi_{22}^{kj}\phi_{11}^{lk})\lambda_{il},\\
\langle e_i^*, x_0 y_m^*\rangle\overset{(c^\prime)}= 0, \quad &\langle x_0 y_m^*, e_i^*\rangle\overset{(d^\prime)}=0,\\
\langle y_1^*, x_0 y_2^*\rangle\overset{(e_1\prime)}=(-1)^dW^\prime, & \langle y_1^*, x_0 y_1^*\rangle\overset{(e_2^\prime)}=(-1)^d(\frac{q}{p}W^\prime-\frac{1}{p}X^\prime), \\
\langle y_2^*, x_0 y_2^*\rangle\overset{(e_3\prime)}=(-1)^dY^\prime, &\langle y_2^*, x_0 y_1^*\rangle\overset{(e_4\prime)}=(-1)^d(\frac{q}{p}Y^\prime-\frac{1}{p}Z^\prime),\\
\langle x_0 y_1^*, y_1^*\rangle\overset{(f_1^\prime)}=\frac{q}{p}, &\langle x_0 y_1^*, y_2^*\rangle\overset{(f_2^\prime)}=1,\\
\langle x_0 y_2^*, y_1^*\rangle\overset{(f_3^\prime)}=-\frac{1}{p}, &\langle x_0 y_2^*, y_2^*\rangle\overset{(f_4^\prime)}=0,\\
\langle y_m^*, \eta_jy_1^*y_2^*\rangle\overset{(g^\prime)}=0, \quad&\langle\eta_jy_1^*y_2^*, y_m^*\rangle\overset{(h^\prime)}=0.
\end{cases}$$
\end{corollary}

\begin{corollary}\label{basises} The vector set $\{\eta_1y_1^*y_2^*, \eta_2y_1^*y_2^*, \cdots, \eta_ny_1^*y_2^*, x_0 y_1^*, x_0 y_2^*\}$ forms a $\mathbbm{k}$-linear basis of $B^!_{d+1}$.
\end{corollary}
\begin{proof} Suppose that:
$$a_1\eta_1y_1^*y_2^*+\cdots+a_n\eta_ny_1^*y_2^*+b_1 x_0 y_1^*+b_2x_0 y_2^*=0$$
for some coefficients  $a_1, \cdots, a_n, b_1, b_2\in \mathbbm{k}$.  For each $i=1, 2, \cdots, n$, we have
\begin{align*} 0=&\langle e_i^*, a_1\eta_1y_1^*y_2^*+\cdots+a_n\eta_ny_1^*y_2^*+b_1 x_0 y_1^*+b_2 x_0 y_2^*\rangle\\
=&\sum\limits_{j=1}^na_j\langle e_i^*, \eta_jy_1^*y_2^*\rangle+b_1\langle e_i^*, x_0 y_1^*\rangle+b_2\langle e_i^*, x_0 y_2^*\rangle\\
=&a_i. \quad \quad\quad\quad \quad\quad\quad \quad\quad\quad\quad\quad\,\quad\quad\text{(by   Equations $(a^\prime)$  and $(c^\prime)$)}
\end{align*}
Similarly,  we have:
\begin{align*} 0=&\langle y_1^*, a_1\eta_1y_1^*y_2^*+\cdots+a_n\eta_ny_1^*y_2^*+b_1 x_0 y_1^*+b_2 x_0 y_2^*\rangle\\
=&\sum\limits_{j=1}^na_j\langle y_1^*, \eta_jy_1^*y_2^*\rangle+b_1\langle y_1^*, x_0  y_1^*\rangle+b_2\langle y_1^*, x_0 y_2^*\rangle\\
=& b_1(-1)^d(\frac{q}{p}W^\prime-\frac{1}{p}X^\prime)+b_2(-1)^dW^\prime, \quad\quad\text{(by  Equations  $(g^\prime)$,  $(e_1^\prime)$ and $(e_2^\prime)$)}.
\end{align*}
and
$$ b_1(-1)^d(\frac{q}{p}Y^\prime-\frac{1}{p}Z^\prime)+b_2(-1)^dY^\prime=0$$
obtained in a similar way.
So we obtain  a system of linear equations:
$$\begin{cases}
(\frac{q}{p}W^\prime-\frac{1}{p}X^\prime)b_{1}+W^\prime b_{2}&=0,\\
(\frac{q}{p}Y^\prime-\frac{1}{p}Z^\prime) b_{1}+ Y^\prime b_{2}&=0.
\end{cases}$$
The determinant of the  matrix $\begin{pmatrix} \frac{q}{p}W^\prime-\frac{1}{p}X^\prime & W^\prime\\\frac{q}{p}Y^\prime-\frac{1}{p}Z^\prime & Y^\prime
\end{pmatrix}$ is nonzero by Lemma \ref{matrixin}. Hence, $b_1=b_2=0$. Thus,
the vectors $\eta_1y_1^*y_2^*, \eta_2y_1^*y_2^*, \cdots, \eta_ny_1^*y_2^*, x_0  y_1^*$ and $x_0 y_2^*$ are linear independent. On the other hand, by Equation $\eqref{hilbert}$, we have $\dim B^!_{d+1}=2\dim A^!_{d}+ \dim A^!_{d-1}=n+2$.
That is, these vectors form a $\mathbbm{k}$-linear basis of $B^!_{d+1}$.
\end{proof}

Now, we are ready to compute the  Nakayama automorphism $\mu$ of the Frobenius algebra $B^!$. This automorphism  is determined by the equation $$\langle a, b\rangle=\langle \mu(b), a\rangle$$
for any $a, b\in B^!$ (see \eqref{nakdef}).
Note  that $B^!$ is generated by the degree $1$ elements:  $e_1^*,
e_2^*, \cdots, e_n^*, y_1^*, y_2^*$.  Hence, we just need  to describe the images of those elements under the Nakayama automorphism.  By Corollary \ref{basises}, we see that \\
$\{\eta_1y_1^*y_2^*, \eta_2y_1^*y_2^*, \cdots, \eta_ny_1^*y_2^*, x_0 y_1^*, x_0 y_2^*\}$ forms a basis of $B^!_{d+1}$.  Due to the fact that the Nakayama automorphism is graded, we can use the equations in Corollary \ref{value}  to determine  the Nakayama automorphism.

\begin{proposition} \label{nakayaab} The restriction of the Nakayama automorphism $\mu_{B^!}$  of $B^!$ to $A^!$ equals $\mu_{A^!}(\det_l\phi)^*$.
\end{proposition}
\begin{proof}
Suppose that
$$\mu_{B^!}(e_i^*)=k_{i, 1}e_1^*+\cdots+k_{i, n}e_n^*+k_{i, n+1}y_1^*+k_{i, n+2}y_2^*.$$
Since $\langle-, -\rangle$ is a Frobenius pair,
$$\langle x_0 y_m^*, e_i^*\rangle=\langle \mu_{B^!}(e_i^*), x_0 y_m^*\rangle$$
for $m=1, 2$.  From  Equations ($d^\prime$) and ($c^\prime$) in Corollary \ref{value},  we obtain:
\begin{align*}
0=&\langle \mu_{B^!}(e_i^*), x_0 y_m^*\rangle\\
=&\langle k_{i, 1}e_1^*+\cdots+k_{i, n}e_n^*+k_{i, n+1}y_1^*+k_{i, n+2}y_2^*, x_0  y_m^*\rangle\\
=&\langle k_{i, n+1}y_1^*+k_{i, n+2}y_2^*, x_0  y_m^*\rangle\\
=&k_{i, n+1}\langle y_1^*, x_0 y_m^*\rangle+k_{i, n+2}\langle y_2^*, x_0 y_m^*\rangle
\end{align*}
From Equations  ($e_1^\prime$)-($e_4^\prime$) in Corollary \ref{value} , we obtain the following system of linear equations:
$$\begin{cases}
&(\frac{q}{p}W^\prime-\frac{1}{p}X^\prime)k_{i, n+1}+(\frac{q}{p}Y^\prime-\frac{1}{p}Z^\prime)k_{i, n+2}=0,\\
&W^\prime k_{i, n+1}+ Y^\prime k_{i, n+2}=0,
\end{cases}$$
Since the determinant of the  matrix $\begin{pmatrix} \frac{q}{p}W^\prime-\frac{1}{p}X^\prime & \frac{q}{p}Y^\prime-\frac{1}{p}Z^\prime \\ W^\prime & Y^\prime
\end{pmatrix}$ is nonzero by Lemma \ref{matrixin},  we have:
$$ k_{i, n+1}=0= k_{i, n+2}$$
for each $i$.
Following  the definition of the Nakayama automorphism (see \eqref{nakdef}) and Equations ($a^\prime$) and ($b^\prime$),  we arrive at:
$$\mu_{B^!}(e_i^*)=\sum\limits_{j}(\sum\limits_{k, l}(\frac{q}{p}\phi_{21}^{ki}\phi_{11}^{lk}\!-\!\frac{1}{p}\phi_{21}^{ki}\phi_{12}^{lk}
\!+\phi_{22}^{ki}\phi_{11}^{lk})\lambda_{il})e_j^*.$$
On the other hand,  we claim that
$$(\det_l\phi)^*(e_i^*)=\sum\limits_{k, l}(\frac{q}{p}\phi_{21}^{ki}\phi_{11}^{lk}\!-\!\frac{1}{p}\phi_{21}^{ki}\phi_{12}^{lk}
\!+\phi_{22}^{ki}\phi_{11}^{lk})e_l^*.$$
Since for any $e_m$,
\begin{align*} (\det_l\phi)^*(e_i^*)(e_m)=&e_i^*(\det_l\phi(e_m))\\
=&e_i^*(\frac{q}{p}\phi_{21}\circ\phi_{11}(e_m)+\phi_{22}\circ\phi_{11}(e_m)-\frac{1}{p}\phi_{21}\circ\phi_{12}(e_m))\\
=&e_i^*(\frac{q}{p}\phi_{21}(\sum\limits_k\phi_{11}^{mk}e_k)+\phi_{22}(\sum\limits_k\phi_{11}^{mk}e_k)-\frac{1}{p}\phi_{21}(\sum\limits_k\phi_{12}^{mk}e_k))\\
=&e_i^*(\frac{q}{p}\sum\limits_{k, l}\phi_{11}^{mk}\phi_{21}^{kl}e_l+\sum\limits_{k, l}\phi_{11}^{mk}\phi_{22}^{kl}e_l-\frac{1}{p}\sum\limits_{k, l}\phi_{12}^{mk}\phi_{21}^{kl}e_l)\\
=&\frac{q}{p}\sum\limits_{k}\phi_{11}^{mk}\phi_{21}^{ki}+\sum\limits_{k}\phi_{11}^{mk}\phi_{22}^{ki}-\frac{1}{p}\sum\limits_{k}\phi_{12}^{mk}\phi_{21}^{ki},
\end{align*}
which coincides the value of $\sum\limits_{k, l}(\frac{q}{p}\phi_{21}^{ki}\phi_{11}^{lk}\!-\!\frac{1}{p}\phi_{21}^{ki}\phi_{12}^{lk}
\!+\phi_{22}^{ki}\phi_{11}^{lk})e_l^*(e_m)$.

It follows that
\begin{align*}
\mu_{A^!}(\det_l\phi)^*(e_i^*)=&\mu_{A^!}(\sum\limits_{k, l}(\frac{q}{p}\phi_{21}^{ki}\phi_{11}^{lk}\!-\!\frac{1}{p}\phi_{21}^{ki}\phi_{12}^{lk}
\!+\phi_{22}^{ki}\phi_{11}^{lk})e_l^*)\\
=&\sum\limits_{k, l}(\frac{q}{p}\phi_{21}^{ki}\phi_{11}^{lk}\!-\!\frac{1}{p}\phi_{21}^{ki}\phi_{12}^{lk}
\!+\phi_{22}^{ki}\phi_{11}^{lk})\mu_{A^!}(e_l^*)\\
=&\sum\limits_{k, l}(\frac{q}{p}\phi_{21}^{ki}\phi_{11}^{lk}\!-\!\frac{1}{p}\phi_{21}^{ki}\phi_{12}^{lk}
\!+\phi_{22}^{ki}\phi_{11}^{lk})\sum\limits_{j}\lambda_{jl}e_j^*.
\end{align*}
That is,  $\mu_{B^!}(e_i^*)=\mu_{A^!}(\det_l\phi)^*(e_i^*)$, for all $i$.
\end{proof}

We need the following technical result although the proof is obvious.
\begin{lemma}  \label{tech} Let $E=\mathbbm{k}\oplus E_1\oplus\cdots\oplus E_m$ be a graded Frobenius algebra which is generated in degree $1$.
Suppose that $\{\alpha_1, \alpha_2\}$ and $\{\beta_1, \beta_2\}$ are $\mathbbm{k}$-linear bases of $E_1$ and $E_{m-1}$ respectively.
Let
$$\begin{cases}
\langle \alpha_1, \beta_1 \rangle= a, \quad &\langle\beta_1, \alpha_1\rangle=e,\\
\langle \alpha_1, \beta_2 \rangle= b, \quad &\langle\beta_2, \alpha_1\rangle=f,\\
\langle \alpha_2, \beta_1 \rangle= c, \quad &\langle\beta_1, \alpha_2\rangle=g,\\
\langle \alpha_2, \beta_2 \rangle= d, \quad &\langle\beta_2, \alpha_2\rangle=h.
\end{cases}$$
Then, the Nakayama automorphism of $E$ is given by:
\begin{align*}\mu(\alpha_1)&=\frac{de-cf}{ad-bc}\alpha_1+\frac{af-be}{ad-bc}\alpha_2,\\
\mu(\alpha_2)&=\frac{dg-ch}{ad-bc}\alpha_1+\frac{ah-bg}{ad-bc}\alpha_2.
 \end{align*}
 \end{lemma}
\begin{proof} Note that the Frobenius pair $\langle-, -\rangle$ is
a nondegenerate bilinear form. It follows that $ad-bc\neq 0$. Since the Nakayama automorphism is graded and $E$ is generated in degree $1$,  the Nakayama automorphism is determined by the assumed equations. we are only to determine the image of elements of degree 1. The conclusion follows from a direct computation.
\end{proof}

\begin{proposition} \label{nakac} The image of $y_1^*$ and $y_2^*$ under Nakayama automorphism $\mu_{B^!}$ are given as follows:
\begin{align*}\mu_{B^!}(y_1^*)=&(-1)^{d+1}\big((qX+\frac{q}{p}X+\frac{1}{p}W)y_1^*+ (qZ+\frac{q}{p}Z+\frac{1}{p}Y)y_2^*\big),\\
\mu_{B^!}(y_2^*)=&(-1)^{d+1}(pXy_1^*+ pZy_2^*).\end{align*}
where $W,X,Y$ and $Z$ form the  homological determinant of  $\sigma$.
\end{proposition}
\begin{proof} The proof is similar to the one of Proposition \ref{nakayaab}.
 Suppose that
$$\mu_{B^!}(y_1^*)=k_{1}e_1^*+\cdots+k_{n}e_n^*+k_{n+1}y_1^*+k_{n+2}y_2^*.$$
Since the equation $\langle \eta_jy_1^*y_2^*, y_1^*\rangle=\langle \mu_{B^!}(y_1^*), \eta_jy_1^*y_2^*\rangle$,   where  $j=1, 2, \cdots, n$,  we have:
\begin{align*} 0=&\langle \eta_jy_1^*y_2^*, y_1^*\rangle\\
=&\langle \mu_{B^!}(y_1^*), \eta_jy_1^*y_2^*\rangle\\
=&\sum\limits_{i=1}^{n}k_i\langle e_i, \eta_jy_1^*y_2^*\rangle +k_{n+1}\langle y_1^*, \eta_jy_1^*y_2^*\rangle+ k_{n+2}\langle y_2^*, \eta_jy_1^*y_2^*\rangle\\
=&\sum\limits_{i=1}^{n}k_i\delta_{ij}=k_j.
\end{align*}
It follows that $\mu_{B^!}(y_1^*)=k_{n+1}y_1^*+k_{n+2}y_2^*$.  Similarly, $\mu_{B^!}(y_2^*)=l_{n+1}y_1^*+l_{n+2}y_2^*$ for some $l_{n+1}, l_{n+2}\in \mathbbm{k}$. Hence, both $\mu_{B^!}(y_1^*)$ and $\mu_{B^!}(y_2^*)$ are completely determined by the values  in   Equations ($e_1^\prime$)-($e_4^\prime$) and ($f_1^\prime$)-($f_4^\prime$) in Corollary \ref{value}.
Thus,   we arrive at the case of Lemma \ref{tech}.  It follows that
\begin{align*}\mu_{B^!}(y_1^*)=&(-1)^d(\frac{qY^\prime+\frac{q}{p}Y^\prime-\frac{1}{p}Z^\prime}{W^\prime Z^\prime-X^\prime Y^\prime}y_1^*+ \frac{-qW^\prime-\frac{q}{p}W^\prime+\frac{1}{p}X^\prime}{W^\prime Z^\prime-X^\prime Y^\prime}y_2^*),\\
\mu_{B^!}(y_2^*)=&(-1)^d(\frac{pY^\prime}{W^\prime Z^\prime-X^\prime Y^\prime}y_1^*+ \frac{-pW^\prime}{W^\prime Z^\prime-X^\prime Y^\prime}y_2^*).
\end{align*}
Finally, the statement follows from  the  equation:
$$\begin{pmatrix} W& X \\ Y& Z
\end{pmatrix}=\dfrac{1}{W^\prime Z^\prime-X^\prime Y^\prime}\begin{pmatrix} Z^\prime & -Y^\prime \\ -X^\prime & W^\prime
\end{pmatrix},$$  a consequence of Lemma \ref{matrixin}.
\end{proof}

\begin{proposition} \label{nakaab}
The restriction of the Nakayama automorphism $\nu_{B}$ of
$B$ to $A$ equals $(\det_r\sigma)^{-1}\nu$, and
$$
\nu_{B}\begin{pmatrix}y_1 \\  y_2
\end{pmatrix}=(\hdet\sigma)\mathbb{P}^{-1}\begin{pmatrix}y_1 \\  y_2
\end{pmatrix},
$$  where $\mathbb{P}$ is given by (\ref{matrixp}).
\end{proposition}
\begin{proof}  By Proposition \ref{nakayaab} and Proposition \ref{nakac}, the restriction of Nakayama automorphism
$\mu_{B^!}$ to  $B^!_1$ has the form  $\begin{pmatrix}  Q_1& 0 \\ 0& Q_2
\end{pmatrix}$, where $Q_1\in M_d(\mathbbm{k})$ and $Q_2\in M_2(\mathbbm{k})$.
By Proposition \ref{criteria} and Equation \ref{eqdual}, the Nakayama automorphism of $B$ is also of this type.
Combining Proposition \ref{criteria}, Proposition \ref{nakayaab} and Proposition \ref{determ}(3) we  obtain the first statement.  By Proposition \ref{criteria} and Proposition \ref{nakac}, we have:
 \begin{align*}\nu_{B}(y_1)=&(qX+\frac{q}{p}X+\frac{1}{p}W)y_1+ pXy_2,\\
\nu_{B}(y_2)=&(qZ+\frac{q}{p}Z+\frac{1}{p}Y)y_1+ pZy_2.\end{align*}
Thus, the second conclusion follows from the definition of the homological determinant of $\sigma$ in Definition \ref{hdett} and Equation \ref{matrixp}.
\end{proof}

Now we are ready to characterize the Calabi-Yau property of a trimmed double Ore extension of a Koszul AS-regular algebra.

\begin{theorem} \label{nakab} Suppose that $A$ is a Koszul AS-regular algebra with Nakayama automorphism $\nu$. Let $B=A_P[y_1, y_2; \sigma]$ be a trimmed  double Ore extension of $A$.
Then $B$ is Calabi-Yau if and only if  $\det_r\sigma=\nu$ and $\hdet\sigma=\mathbb{P}$.
\end{theorem}
\begin{proof} Since $B$ is Koszul and is of finite global dimension,  the Koszul $B^e$-bimodule complex provides a finitely
generated  projective resolution of $B$ of finite length. That is, $B$ is homologically smooth. Because $B$ is connected graded, its only  inner automorphism is the identity. So for $B$ to be Calabi-Yau, its Nakayama automorphism must be the identity.  Therefore, the statement is a consequence of Proposition \ref{nakaab}.
\end{proof}

\begin{remark} \label{existunique}
For a Koszul AS-regular algebra $A$ with Nakayama automorphism $\nu$, there exists a unique skew polynomial extension $D$ such that $D$ is Calabi-Yau,  see \cite{GK14, GYZ14, HVZ13, LWW12, RRZ14}. Here, we  consider
the existence and the uniqueness of a Calabi-Yau trimmed double Ore extension of a Koszul AS-regular algebra.

$(1)$. For any Koszul AS-regular algebra $A$ with Nakayama automorphism $\nu$,  consider the trimmed double Ore extension
$B=A_P[y_1, y_2; \sigma]$ with $P=(1, 0)$ and $
\sigma=\begin{pmatrix} \nu & 0 \\ 0 & id
\end{pmatrix}$. Then $B$ is Calabi-Yau. But it is easy to see that $B$ is an iterated Ore extension
of $A$ (see \cite[Proposition 3.6]{ZZ09} or its proof). Hence, we ask if there exists a nontrivial
double Ore extension $B$ (not an iterated one) such that $B$ is Calabi-Yau?  The answer is negative from the following example.

Let $A=\mathbbm{k}\langle x_1, x_2\rangle/(x_2x_1-x_1x_2-x_1^2)$ be the Jordan plane. Its Nakayama automorphism $\nu$ is given by
$\nu(x_1)=x_1$ and $\nu(x_2)=2x_1+x_2$. Then,
there is only one nontrivial double Ore extension by the classification in \cite{ZZ09}, namely,
the type $\mathbb{H}:=A_P[y_1, y_2; \sigma]$ with $P=(-1, 0)$ and $\sigma$  given by the matrix
$\begin{pmatrix} 0 & h& 0 & 0 \\
 h & 0& 0 & 0 \\
 0 & hf& 0 & h \\
 hf & 0 & h & 0
\end{pmatrix}$ with $0\neq h\in \mathbbm{k}$ and $f\in \mathbbm{k}$.
Now, $\det_r(\sigma)$ is an automorphism given by $\det_r(\sigma)(x_1)=h^2x_1$ and $\det_r(\sigma)(x_2)=2h^2fx_1+h^2x_2$.
Let $x_0$ to be a base element of the 1-dimensional space $A^!_2$. Then
$$\sigma^*(x_0)=\begin{pmatrix} h^{2}x_0 & 0 \\
 0 & h^{2}x_0
\end{pmatrix}.$$
That is, $W=h^{2}, X=0, Y=0$ and $Z=h^{2}$.
By Proposition \ref{nakaab}, the Nakayama automorphism of  $\mathbb{H}$ is
\begin{align*} \nu:& x_1\to h^{-2}x_1\\
&x_2\to h^{-2}((2-2f)x_1+x_2)\\
& y_1\to -h^{2}y_1\\
&y_2 \to -h^{2}y_2.
\end{align*}
Therefore, there is no Calabi-Yau algebra in the class of the type $\mathbb{H}$.

$(2)$. For the uniqueness, let $A=\mathbbm{k}\langle x_1, x_2\rangle/(x_2x_1+x_1x_2)$ be the quantum plane whose Nakayama automorphism is given by $\nu(x_1)=-x_1$ and $\nu(x_2)=-x_2$. Suppose that $B:=A_P[y_1, y_2; \sigma]$ with $P=(-1, 0)$, where $\sigma$ is given by the matrix
$\begin{pmatrix} 0 & 0 & -g & f \\
 0 & 0 & f & -g \\
 g & f & 0 & 0 \\
 f & g & 0 & 0
\end{pmatrix}$ with $f, g\in \mathbbm{k}$ and $f^2\neq g^2$. So $B$ is of type $\mathbb{N}$ in the classification of \cite{ZZ09}.
Now, $\det_r(\sigma)$ is an automorphism given by $\det_r(\sigma)(x_1)=(f^2-g^2)x_1$ and $\det_r(\sigma)(x_2)=(f^2-g^2)x_2$.  Let $x_0$ be a base element of the 1-dimensional space $A^!_2$. Then we have:
$$\sigma^*(x_0)=\begin{pmatrix} (f^2-g^2)x_0 & 0 \\
 0 & (f^2-g^2)x_0
\end{pmatrix}.$$
In this case, $W=f^2-g^2, X=0, Y=0$ and $Z=f^2-g^2$.
Thus, the Nakayama automorphism of $B$ is
equal to $(g^2-f^2)\id$ by Proposition \ref{nakaab}.
Hence, $B$ is Calabi-Yau if and only if $g^2-f^2=1$. Therefore,
a trimmed double Ore extension, which is Calabi-Yau, of a Koszul AS-regular algebra may not be unique if it exists.

\end{remark}

\begin{remark} In the first example in Remark \ref{existunique}, we know that $\det_r\sigma=\nu_A$ if and only if $h^2=f=1$. Moreover, $\hdet\sigma=\begin{pmatrix} h^{2} & 0 \\
 0 & h^{2}
\end{pmatrix}.$ But, $\mathbb{P}=\begin{pmatrix} -1 & 0 \\
 0 & -1
\end{pmatrix}$.
Therefore,  the condition $\det_r\sigma=\nu_A$ and the condition $\hdet\sigma=\mathbb{P}$
 in Theorem \ref{nakab} are independent. More examples can be constructed from iterated Ore extensions, see Example \ref{exhdet}.
\end{remark}

To end this section, we return to discuss the Nakayama automorphism and the Calabi-Yau property of
the skew polynomial extension. For a twisted Calabi-Yau algebra $A$ with Nakayama automorphism $\nu$, it
was proved in \cite[Theorem 3.3]{LWW12} that the Nakayama automorphism of an Ore extension $D=A[t; \theta, \delta]$ is given by
$$\nu_{D}(x)=\begin{cases} \theta^{-1}\circ\nu(x), & x \in A;
  \\
 ax+b, & x=t,
\end{cases}$$
for some $a, b\in A$ with $a$ invertible. It was also remarked there that if $\delta=0$, then $\nu_{D}(t)=at$.
Now if we restrict to Koszul algebras, we can describe the Nakayama automorphism more explicitly as follows.

\begin{proposition} \label{nakare}  Suppose that $A$ is a Koszul AS-regular algebra with Nakayama automorphism $\nu$, $\theta$ is a graded algebra automorphism of $A$ and $D=A[t; \theta]$.
The  Nakayama automorphism $\nu_{D}$ of
$D$  is given by:
$$\nu_{D}(x)=\begin{cases} \theta^{-1}\circ\nu(x), & x \in A
  \\
 (\hdet\theta)x, & x=t.
\end{cases}$$
\end{proposition}
\begin{proof} We only give a sketch of the proof since it is similar to the one of Proposition \ref{nakaab}.
Suppose that $D=T(V \oplus \mathbbm{k} \, t)/\langle \, R_D \rangle$. The generating relations in $D$ are of two
types: $te_i-\theta(e_i)t\, (1 \leq i \leq n)$ and the relations from $A$.
Obviously, $\{e_1^*,
e_2^*, \cdots, e_n^*, t^*\}$ forms a $\mathbbm{k}$-linear basis for $D_1^!$.
By \cite[Proposition
2.4]{LSV96},  the defining relations for $D^!$ consist of the following three types:
\begin{enumerate}
\item [(1)] the relations from $A^!$;

\item [(2)] $\{t^*e_i^*+(\theta^{-1})^*(e_i^*)t^* \mid 1 \leq i \leq n\}$;

\item [(3)] $\{(t^*)^2\}$.
\end{enumerate}

By \cite[Proposition 2.5]{LSV96}, $D^!$ is a free $A^!$-module with
basis $\{1, t^*\}$. Hence, $x_0 t^*$ is  a base element of the
$1$-dimensional space $D_{d+1}^!$, denoted  $\varepsilon$, where $x_0$ is  a base element of  the $1$-dimensional $\mathbbm{k}$-space $A_d^!$.   Now let
$(b_{ij})_{n\times n}$ be the matrix of the restriction of $\theta^{-1}$ to $V$, i. e.,
\begin{equation} \label{nakainv} \theta^{-1}(e_i)=\sum_j b_{ij}e_j
\end{equation}
 for each $i$.
Then, we have
\begin{enumerate}
\item [(1)]\,\,\, $\{\eta_1t^*, \eta_2t^*, \cdots, \eta_nt^*, x_0\}$ is a $\mathbbm{k}$-linear
basis of $D^!_d$;
\item [(2)]\,\,\, the  following equations hold:
$$\begin{cases}
\, e_i^*\eta_jt^*=\delta_{ij}\varepsilon, \quad\quad\quad\quad
&\quad \eta_it^*e_j^*=-\sum_kb_{kj}\lambda_{ik}\varepsilon,\\
\, e_i^*x_0=0, & \quad x_0 e_i^*=0,\\
\, t^*x_0=(-1)^d(\hdet(\theta))^{-1}\varepsilon,  & \quad x_0 t^*=\varepsilon,\\
\, t^*\eta_jt^*=0, & \quad\eta_jt^*t^*=0.
\end{cases}$$
\end{enumerate}
Using  the same argument in the proof of Proposition \ref{nakayaab} and Proposition \ref{nakac}, one obtains that  the Nakayama automorphism $\mu_{D^!}$ of
$D^!$  is given by:
$$\mu_{D^!}(\alpha)=\begin{cases} -\mu_{A^!}\circ(\theta^{-1})^*(\alpha), & \alpha \in A^!
  \\
 (-1)^d(\hdet\theta)\alpha, & \alpha=t^*.
\end{cases}$$
The last step is to transfer $\mu_{D^!}$  to the Nakayama automorphism $\nu_{D}$ of
$D$ by  Proposition \ref{criteria}.
\end{proof}

Note that the homological determinant of the Nakayama automorphism of a Koszul AS-regular algebra is equal to $1$ \cite[Theorem 0.4]{RRZ14}. Thus, we arrive at the following result which was proved in \cite{GK14, GYZ14, HVZ13, LWW12, RRZ14}:

\begin{theorem} \label{orecy}
Suppose that $A$ is a Koszul AS-regular algebra with Nakayama automorphism $\nu$, $\theta$ is a graded algebra automorphism of $A$ and $D=A[t; \theta]$. Then, $D$ is Calabi-Yau if and only if $\theta=\nu$. \qed
\end{theorem}

\section{Skew Laurent Extensions}
In this section, we consider the Calabi-Yau property of the Ore localizations of both $A[t; \theta]$ and  $A_P[y_1, y_2; \sigma]$ with some conditions. For a skew polynomial extension $A[t; \theta]$ of an algebra $A$, the multiplicatively closed set $\{t^i; i\in \mathbb{N}\}$
is an Ore set. The localization of $A[t; \theta]$ with respect to this Ore set is just the skew Laurent polynomial extension $A[t^{\pm 1}; \theta]$.  Farinati proposed a general notion of a noncommutative localization in  \cite{F05}. It was  proved there that the Van den Bergh duality is preserved by such a localization and the corresponding dualizing module is also explicitly described. The Ore localization is an example of a noncommutative localization \cite[Example 8]{F05}.

\begin{proposition} \label{nakal} Suppose that
$A$ is a Koszul AS-regular algebra of dimension $d$ and  $D=A[t; \theta]$
is a skew polynomial extension of $A$. Then, the Nakayama automorphism $\widetilde{\nu}$ of $A[t^{\pm 1}; \theta]$
is given by
$$\widetilde{\nu}(x)=\begin{cases} \nu_D(x), & x \in D
  \\
 \frac{1}{\hdet\theta} x, & x=t^{-1}.
\end{cases}$$
\end{proposition}
\begin{proof} By assumption and \cite[Theorem 6]{F05}, we have:
$$\Ext^i_{E^e}(E, E^e)\cong\begin{cases}0, & i \neq d+1
  \\
  E\otimes_DD^{\nu}\otimes_DE(d+1), & i=d+1,
\end{cases}$$
where $E$ stands for the algebra $A[t^{\pm 1}; \theta]$. Thus, the claim follows from the description of the Nakayama automorphism of $D$ in Proposition \ref{nakare}.
\end{proof}

\begin{theorem} \label{skewl} Suppose that $A$ is a Koszul AS-regular algebra with  Nakayama automorphism $\nu$ and $\theta$ is a graded algebra automorphism of $A$.
Then, $A[t^{\pm 1}; \theta]$
is graded Calabi-Yau if and only if there exists an integer $n$ such that
$\theta^n=\nu$ and the homological determinant of $\theta$ equals $1$.
\end{theorem}
\begin{proof} It follows from the proof of \cite[Theorem 6]{F05} that $A[t^{\pm 1}; \theta]$ is homologically smooth. Thus, the proof focuses on the description of the Nakayama automorphisms of algebras $A[t; \theta]$ and $A[t^{\pm 1}; \theta]$  as showed in Proposition \ref{nakare} and Proposition \ref{nakal} respectively.  Note that the only invertible elements in $\mathbbm{k}[t^{\pm 1}]$
are monomials. Suppose that  $A[t^{\pm 1}; \theta]$
is Calabi-Yau. Then, its Nakayama automorphism $\widetilde{\nu}$ is inner, i.e., there exists an integer
$n\in \mathbb{Z}$ such that  $\widetilde{\nu}(x)=t^nxt^{-n}$ for each $x\in A[t^{\pm 1}; \theta]$.
In particular, $\widetilde{\nu}(t)=t$. Therefore, $\hdet(\theta)=1$ by Proposition \ref{nakare}. If $n$ is nonnegative, then for each $x\in A$ we have
\begin{align*} \widetilde{\nu}(x)&=\theta^{-1}\nu(x)=t^nxt^{-n}\\
&=t^n(t^{-1}\theta(x)t)t^{-n}\\
&=t^{n-1}\theta(x)t^{1-n}\\
&=\cdots=\theta^n(x).
\end{align*}
Hence, $\nu(x)=\theta^{n+1}(x)$.
Similarly, the claim also holds for the case when $n$ is a negative integer.

Conversely, if $\theta^n=\nu$ for some integer $n$ and the homological determinant of $\sigma$ equals $1$,
then $\widetilde{\nu}(t)=t$. Next, for each $x\in A$, we have
$$\widetilde{\nu}(x)=\theta^{-1}\nu(x)=\theta^{n-1}(x).$$
But in $A[t^{\pm 1}; \theta]$, $\theta(x)=txt^{-1}$. That is, both $\theta$ and its inverse are inner.
Therefore, $\widetilde{\nu}$ is an inner automorphism. The proof is completed.
\end{proof}

\begin{example} Let $A=\mathbbm{k}\langle x, y\rangle/(yx-xy-x^2)$ be the Jordan plane. It is a twisted Calabi-Yau
algebra of dimension $2$ whose Nakayama automorphism $\nu$ is given by $\nu(x)=x$ and $\nu(y)=2x+y$. Then, $A[t; \theta]$ is Calabi-Yau if and only if $\theta=\nu$ by Theorem \ref{orecy}. It is not hard to see that each graded automorphism $\theta$ of $A$ has the form $\theta(x)=ax$ and $\theta(y)=bx+ay$
for some $a, b\in \mathbbm{k}$. By Proposition \ref{hdk}, the homological determinant of $\theta$ is equal to $a^2$. Thus, $A[t^{\pm 1}; \theta]$ is Calabi-Yau if and only if $\theta$ is either given by
$$\begin{cases}\theta(x)=x
  \\
  \theta(y)=\frac{2}{n}x+y
\end{cases}$$ for some nonzero integer $n$, or given by $$\begin{cases}\theta(x)=-x
  \\
  \theta(y)=\frac{2}{n}x-y
\end{cases}$$ for some even integer $n$.
\end{example}

Finally, we consider  the localization or the quotient ring of the double Ore extension $B$ with respect to
the Ore set generated by new generators. However, we can only do this in some special case as follows.

\begin{proposition} \label{iterated} Let $B=A_P[y_1, y_2; \sigma]$ be a trimmed double Ore extension
with $P=(p, 0)$ and $
\sigma=\begin{pmatrix} \tau & 0 \\ 0 & \xi
\end{pmatrix}$.
Then,
\begin{enumerate}
\item [(1)]
Both $\tau$ and $\xi$ are automorphisms of $A$. Moreover, they commutate with each other.

\item [(2)] The multiplicatively closed set $S:=\{ay_1^{n_1}y_2^{n_2}; a\in k, n_1, n_2\in \mathbb{Z}_{\geq 0}\}$ is
an Ore set.

\item [(3)] The quotient ring $B_S$ of $B$ with respect  to $S$ exists.
\end{enumerate}
\end{proposition}
\begin{proof} Since $B$ is a trimmed double Ore extension of $A$, $\sigma$ is invertible according to Lemma \ref{doeinv}. Hence, both $\tau$ and $\xi$ are automorphisms of $A$. By the definition of the right determinant of $\sigma$ (see \eqref{detdef}) and its equivalent description in Proposition \ref{determ}, we have $\tau\xi=\xi\tau$.  The rest of the proof is straightforward.
\end{proof}

In fact, the algebra $B=A_P[y_1, y_2; \sigma]$ considered above is an iterated skew polynomial extension $A[y_1; \tau][y_2; \xi^\prime]$ where
$\xi^\prime$ is the automorphism of  $A[y_1; \tau]$ defined as follows
$$\xi^\prime(x)=\begin{cases} \xi(x), & x \in A;
  \\
 px, & x=y_1.
\end{cases}$$
If $p=1$, then the quotient ring $B_S$ is isomorphic to the iterated skew Laurent ring $A[y_1^{\pm 1}, y_2^{\pm 1}; \tau, \xi]$ (see \cite[p.23-24]{GW04}). In the case of $p\neq 1$, we can also construct the iterated skew Laurent ring, denoted $A_P[y_1^{\pm 1}, y_2^{\pm 1}; \tau, \xi]$ or just $A_P[y_1^{\pm 1}, y_2^{\pm 1}; \sigma]$. Similarly, the quotient ring $B_S$ in the above Proposition is isomorphic to the iterated skew Laurent ring $A_P[y_1^{\pm 1}, y_2^{\pm 1}; \sigma]$

\begin{theorem}  \label{skew2} Suppose that $A$ is a
Koszul AS-regular algebra with Nakayama automorphism $\nu$,  $B=A_P[y_1, y_2; \sigma]$ is a trimmed double Ore extension
with $P=(p, 0)$ and $
\sigma=\begin{pmatrix} \tau & 0 \\ 0 & \xi
\end{pmatrix}$ and $B_S=A_P[y_1^{\pm 1}, y_2^{\pm 1}; \sigma]$.
Then, $B_S$
is Calabi-Yau if and only if there exist two integers  $m, n$ such that the following conditions are satisfied:
\begin{enumerate}
\item [(1)]
 $\tau^n\xi^m=\nu$;

\item [(2)] $\hdet(\tau)=p^{m}$ and $\hdet(\xi)=1/p^{n}$.
\end{enumerate}
\end{theorem}
\begin{proof} The homologically smoothness of $B_S$ also follows from the proof of \cite[Theorem 6]{F05}. Observe that for the given homomorphism $\sigma: A\to M_{2\times 2}(A)$, the induced algebra homomorphism $\sigma^*$ form $A^!$ to  $M_{2\times 2}(A^!)$ has the form $\begin{pmatrix} \tau^* & 0 \\ 0 & \xi^*
\end{pmatrix}$, where $\tau^*$ and $\xi^*$ are automorphisms of $A^!$ induced by $\tau$ and $\xi$ respectively. By Example \ref{exhdet} and  Proposition \ref{nakaab} we obtain that the Nakayama automorphism of $B$ is given as follows:
$$\nu_{B}(x)=\begin{cases} (\tau\xi)^{-1}\circ\nu(x), & x \in A;
  \\
 \frac{1}{p}(\hdet\tau)x, & x=y_1;\\
 p(\hdet\xi)x, & x=y_2.
\end{cases}$$
Thus, it follows from \cite[Theorem 6]{F05} that the Nakayama automorphism $\widetilde{\nu}$ of $B_S$
is given by
$$\widetilde{\nu}(x)=\begin{cases} \nu_B(x), & x \in B
  \\
 \frac{p}{\hdet\tau} x, & x=y_1^{-1}\\
  \frac{1}{p\hdet\xi} x, & x=y_2^{-1}
\end{cases}$$

Note that the only invertible elements in $B_S$
are monomials $a_{nm}y_1^ny_2^m$ for some $a_{nm}\in \mathbbm{k}$ and $n, m\in \mathbb{Z}$.
Suppose that  $B_S$
is Calabi-Yau. Then, its' Nakayama automorphism $\widetilde{\nu}$ is inner, i.e., there exists integer
$m, n\in \mathbb{Z}$ such that  $\widetilde{\nu}(x)=y_1^ny_2^mxy_2^{-m}y_1^{-n}$ for each $x\in B_S$.
In particular, $\widetilde{\nu}(y_1)=y_1^ny_2^my_1y_2^{-m}y_1^{-n}=\frac{1}{p}(\hdet\tau)y_1$. It follows that
$\hdet(\tau)=p^{m+1}$ since  $y_1$ and $y_2$ satisfy $y_2y_1=py_1y_2$.  Similarly, we  have
$\hdet(\xi)=1/p^{n+1}$. Now, without loss of generality,  we may assume that both $n$ and $m$ are nonnegative. For any element $x\in A$,  we have
\begin{align*}
(\tau\xi)^{-1}\circ\nu(x)&=\widetilde{\nu}(x)\\
&=y_1^ny_2^mxy_2^{-m}y_1^{-n}\\
&=y_1^ny_2^{m-1}\xi(x)y_2^{1-m}y_1^{-n}\\
&=\cdots\\
&=y_1^n\xi^m(x)y_1^{-n}\\
&=\cdots\\
&=\tau^n\xi^m(x).
\end{align*}
Hence, $\nu=\tau^{n+1}\xi^{m+1}$.

The proof of the converse is similar.
\end{proof}

In general,  if  $\theta_1, \cdots,  \theta_m$ are commuting graded automorphisms of $A$, one  can construct an iterated skew polynomial
extension $A[y_1, \cdots, y_m; \theta_1, \cdots, \theta_m]$ as follows.  Let $R_1=A[y_1; \theta_1]$. Then, extend $\theta_2$ to an algebra automorphism $\theta_2^\prime$ of
$R_1$ such that $\theta_2^\prime|_A=\theta_2$ and $\theta_2^\prime(y_1)=y_1$. Now let $R_2=A[y_1; \theta_1][y_2; \theta_2^\prime]$. In this way, one can construct $R_i$ for $i=1,2,\cdots m$, such that, for $i<m$,
$R_{i+1}=R_i[y_{i+1}, \theta_{i+1}^\prime]$ , where $\theta_{i+1}^\prime$ is the automorphism of $R_i$ satisfying $\theta_{i+1}^\prime|_A=\theta_{i+1}$ and $\theta_{i+1}^\prime(y_j)=y_j$ for $j=1, \cdots, i$. Finally, let $$R_m=A[y_1; \theta_1][y_2; \theta_2^\prime]\cdots[y_m; \theta_m^\prime].$$

In order to describe the basic data that determine $R_m$, one writes $R_m$ in a different way as follows:
$$R_m=A[y_1, \cdots, y_m; \theta_1, \cdots, \theta_m].$$
Note that $y_iy_j=y_jy_i$,  $y_ia=\theta_i(a)y_i$ for all $i, j$ and any $a\in A$.

Now, let $R=R_m$ for some positive $m$. The quotient ring $R_S$ of $R$ with respect to the multiplicatively closed set $S:=\{y_1^{n_1}\cdots y_m^{n_m}; n_1, \cdots, n_m\in \mathbb{Z}_{\geq 0}\}$ exists and is isomorphic to the iterated skew Laurent ring $A[y_1^{\pm 1}, \cdots, y_m^{\pm 1}; \theta_1, \cdots, \theta_m]$. For more details, we refer to \cite[p.23-24]{GW04}. In the following, we will give a criterion for such an iterated skew polynomial extension of a Koszul AS-regular algebra to be Calabi-Yau.

\begin{theorem} \label{cyite} Suppose that $A$ is a Koszul AS-regular algebra with Nakayama automorphism $\nu$, $R=A[y_1, \cdots, y_m; \theta_1, \cdots, \theta_m]$ and $S:=\{y_1^{n_1}\cdots y_m^{n_m}; n_1, \cdots, n_m\in \mathbb{Z}_{\geq 0}\}$.
Then,
\begin{enumerate}
\item  the Nakayama automorphism $\nu_R$ of $R$ is given by
$$\nu_{R}(x)=\begin{cases} (\theta_1\circ\cdots\circ\theta_m)^{-1}\circ\nu(x), & x \in A
  \\
 (\hdet\theta_i)x, & x=y_i, 1\leq i\leq m;
\end{cases}$$
\item  $R$ is Calabi-Yau if and only if $\theta_1\circ\cdots\circ\theta_m=\nu$ and $\hdet\theta_i=1$ for  all $i$;
\item  $R_S$ is Calabi-Yau if and only if \\
$\mathrm{(i)}$ $\hdet(\theta_i)=1$ for all $i$, and\\
$\mathrm{(ii)}$ there exist  integers   $k_1, \cdots, k_m$ such that
 $\theta_1^{k_1}\cdots\theta_m^{k_m}=\nu$.
\end{enumerate}
\end{theorem}
\begin{proof} It is well-known that a  skew polynomial extension of a Koszul algebra is again Koszul, c.f. \cite[Corollary 1.3]{Ph12}).  So both $R$ and $R_S$ are homologically smooth.  By Proposition \ref{nakare},
the  Nakayama automorphism $\nu_{R_2}$ of
$R_2$  is given by
$$\nu_{R_2}(x)=\begin{cases} (\theta_2^\prime)^{-1}\circ\nu_{R_1}(x), & x \in R_1
  \\
 (\hdet\theta_2^\prime)x, & x=y_2.
\end{cases}$$
It follows from the construction of $\theta_2^\prime$ and the description of the Nakayama automorphism $\nu_{R_1}$ of
$R_1$ that
$$\nu_{R_2}(x)=\begin{cases} (\theta_2\theta_1)^{-1}\circ\nu(x), & x \in A;
  \\
 (\hdet\theta_1)x, & x=y_1;\\
 (\hdet\theta_2^\prime)x, & x=y_2.
\end{cases}$$
On the other hand, according to the proof of Theorem \ref{skew2},
 $\nu_{R_2}(y_2)=(\hdet\theta_2)y_2$.
Hence, $\hdet\theta_2^\prime=\hdet\theta_2$. Repeating this process, we obtain Part (1).  Part (2) follows from Part (1).
The proof of Part (3) is similar to the proof of Theorem \ref{skew2}.
\end{proof}

Note that a typical example of $R_S$ is the smash product of a Koszul AS-regular algebra with a free abelian group algebra.
For example, those Hopf  algebras in the classification of Calabi-Yau pointed Hopf algebras of finite Cartan type in \cite{YZ11}.

%Suppose that $A$ is a Koszul AS-regular algebra and $\sigma, \tau$ are two graded automorphism of $A$.
%Let $B$ be the ring generated over $A$ by two elements $g$ and $h$ subjecting to the relations
%$ga=\theta(a)g$, $ha=\tau(a)h$ and $gh=hg$. Denoted it by $A[g, h; \theta, \tau]$. If $\sigma$ and $\tau$ are commutative, then $B$ is an iterated
%ore extension of $A$ and therefore a Koszul AS-regular algebra. By similar process of Proposition \ref{nakare}
%and Corollary \ref{nakad},  the restriction of the Nakayama automorphism $\nu$ of
%$B$ to $A$ equals to $\theta^{-1}\tau^{-1}\nu_{A}$ and $\nu(g)=\hdet(\theta)g$, $\nu(h)=\hdet(\tau)h$.
%Hence, $A[g, h; \theta, \tau]$ is Calabi-Yau if and only if $\theta\tau=\nu_{A}$ and $\hdet(\theta)=\hdet(\tau)=1$.
%The multiplication closed subset $\{g^ih^j; i, j\in \mathds{Z}^+\}$ in $B$ is a right Ore set. Then, the
%right quotient ring of $B$ exists and it is just $A[g^{\pm}, h^{\pm}; \theta, \tau]$. Furthermore,
%$A[g^{\pm}, h^{\pm}; \theta, \tau]$ is Calabi-Yau if and only if  $\theta^m\tau^n=\nu_{A}$ for some integers  $m$ and $n$ and
%$\hdet(\theta)=\hdet(\tau)=1$.

%GENERAL question: when $A[g^\pm, \theta][h^\pm, \tau]$ is Calabi-Yau for a Koszul AS-regular algebra $A$?

\section*{Acknowledgments}
This work is supported by Natural Science Foundation of China \#11201299 and by an FWO grant.

\bibliography{}

\end{document}